\documentclass[12pt]{amsart}

\usepackage{xcolor}
\usepackage{verbatim}
\usepackage{amsthm}
\usepackage{amssymb}
\PassOptionsToPackage{normalem}{ulem}
\usepackage{ulem}
\usepackage{graphicx}
\usepackage[all]{xy}

\makeatletter

\numberwithin{equation}{section}
\numberwithin{figure}{section}

\theoremstyle{plain}
\newtheorem{thm}{Theorem}[section]
  \theoremstyle{definition}
  \newtheorem{defn}[thm]{Definition}
\theoremstyle{remark}
  \newtheorem{rem}[thm]{Remark}
  \theoremstyle{plain}
  \newtheorem{prop}[thm]{Proposition}
 \theoremstyle{definition}
  \newtheorem{example}[thm]{Example}
  \theoremstyle{plain}
  \newtheorem{cor}[thm]{Corollary}
\newtheorem{lem}[thm]{Lemma}

\def\quot{/\!\!/}
\renewcommand{\hom}{\mathsf{Hom}}
\newcommand{\fun}{\mathsf{Fun}}

\newcommand{\C}{\mathbb{C}}
\newcommand{\X}{\mathfrak{X}}
\newcommand{\R}{\mathbb{R}}
\newcommand{\tr}{\mathrm{tr}}

\newcommand{\GL}{\mathsf{GL}}
\newcommand{\SL}{\mathsf{SL}}

\newcommand{\SU}{\mathsf{SU}}
\newcommand{\rep}{\mathcal{R}}
\newcommand{\moduli}{\mathcal{M}}
\newcommand{\fnc}{\mathcal{F}}
\newcommand{\gauge}{\mathcal{G}}
\newcommand{\V}{\mathcal{V}}
\newcommand{\ab}{\mathbf{a}}
\newcommand{\bb}{\mathbf{b}}
\newcommand{\cb}{\mathbf{c}}
\newcommand{\db}{\mathbf{d}}
\newcommand{\eb}{\mathbf{e}}
\newcommand{\fb}{\mathbf{f}}
\newcommand{\hb}{\mathbf{h}}
\newcommand{\gb}{\mathbf{g}}
\newcommand{\id}{I}
\newcommand{\norm}[1]{|\!|#1 |\!|}

\author[C. Florentino]{Carlos Florentino}

\address{Departamento Matem\'atica, Instituto Superior T\'ecnico, Av. Rovisco
Pais, 1049-001 Lisbon, Portugal}

\email{cfloren@math.ist.utl.pt}

\author[S. Lawton]{Sean Lawton}

\address{Department of Mathematics, The University of Texas-Pan American,
1201 West University Drive Edinburg, TX 78539, USA}

\email{lawtonsd@utpa.edu}



\makeatother

\begin{document}

\title[Character Varieties and Quiver Representations]{Character Varieties and Moduli of Quiver Representations}

\begin{abstract}
Let $G$ be a Lie group and $Q$ a quiver with relations.  In this paper, we define $G$-valued representations of $Q$ which directly generalize $G$-valued representations of finitely generated groups.  Although as $G$-spaces, the $G$-valued quiver representations are more general than $G$-valued representations of finitely generated groups, we show by collapsing arrows that their quotient spaces are equivalent.  We then establish a general criterion for the moduli of $G$-valued quiver representations with relations to admit a strong deformation retraction to a compact quotient by pinching vertices on the quiver.  This provides two different generalizations of main results in \cite{FlLa}.  Lastly, we establish quiver theoretic conditions for the moduli spaces of $\mathsf{GL}(n,\C)$ and $\mathsf{SL}(n,\C)$-valued quiver representations to embed into traditional moduli spaces of quiver representations having constant dimension vector. 
\end{abstract}

\maketitle

\section{Introduction and Motivation}

The study of moduli spaces of representations of finitely generated groups, also known as character varieties, has received intense attention over the last 40 years.  Influential articles include \cite{HN}, \cite{DR}, \cite{AB}, \cite{Hi}, and \cite{G10}.  Some recent advances directly concerning the topology of these moduli spaces include \cite{BGG}, \cite{Holi1},\cite{HoLi2}, \cite{FlLa}, \cite{Ba0}, and \cite{DWW}.  

Let $\Gamma$ be a finitely presented group, and let $G$ be a Lie group.  For our considerations, $G$ will be a compact Lie group $K$ or its complexification $K_\C$ (a complex reductive Lie group). The group of inner automorphisms of $G$ acts on the space of homomorphisms $\hom(\Gamma,G)$.  With respect to this action we consider the GIT quotient $\X_{\Gamma}(K_\C)=\hom(\Gamma,K_\C)\quot K_\C$ and the orbit space $\X_\Gamma(K)=\hom(\Gamma,K)/K$.  In either case, these moduli spaces are referred to as character varieties and, loosely speaking, parametrize flat principal bundles over a topological space $M$ where $\Gamma=\pi_1(M)$.

The topology of moduli spaces of quiver representations has also been recently considered (see \cite{Hau1,Hau2}). For additional work concerning moduli spaces of quiver representations see \cite{Rei,Rei2,LeB,Kin,Cr,CBP}.  A quiver $Q$ is a finite directed graph.  To obtain a (additive) quiver representation associate to each vertex a finite dimensional vector space and to each arrow a linear transformation.  Denote the vertex set of $Q$ by $Q_V$, and let $d = (d_v | v\in Q_V) \in \mathbb{N}^{Q_V}$ be a fixed dimension vector.  Denote the set of arrows by $Q_A$ and each arrow $a$ by $v\to w$ where $v,w\in Q_V$ (with head $h_a=w$ and tail $t_a=v$).  For each such dimension vector fix $\C$-vector spaces $W_v$ of dimension $d_v$ for all $v \in Q_V$. Consider the affine $\C$-space $$\mathcal{R}_d(Q) = \bigoplus_{v\to w\in Q_A}\hom_{\C}(W_{v},W_{w}).$$  The reductive linear algebraic group $$G_d = \prod_{v \in Q_V}\mathsf{GL}(W_v)$$ acts on $\mathcal{R}_d(Q)$ via:  $$(...,g_v,...) \cdot (...,T_{v\to w},...) = (...,
g_wT_{v\to w}g_v^{-1},...).$$  The moduli space of quiver representations is then the GIT quotient of $\mathcal{R}_d$ by $G_d$; namely, $\mathcal{M}_d=\mathcal{R}_d\quot G_d$.

It is the purpose of this paper to define $G$-valued quiver representations (with relations), and discuss the topology of the moduli space of these objects.  The 
$G$-valued quiver representations simultaneously generalize the moduli spaces of principal $G$-bundles (character varieties) and relate to the additive quiver 
representations just discussed.  Indeed, they can be viewed as the equidimensional case of a multiplicative theory of quiver representations.

Define $G$-valued representations of $Q$, denoted by $\fun(Q_A,G)$, to be the collection of set mappings $Q_A\to G$, and define $\gauge_G(Q)$ to be the set mappings $Q_V\to G$.  Then $\gauge_G(Q)$ acts on $\fun(Q_A,G)$ by $g\cdot f(a)=g(h_a)f(a)g(t_a)^{-1}$.  The moduli space of these objects is 
then a categorical quotient via this action:  the GIT quotient if $G$ is complex reductive, or the orbit space if $G$ is compact.  Either way, the quotient is a semi-algebraic set and so naturally is a subset of an affine space.  We consider the subspace topology from such an affine embedding (up to homeomorphism, it is independent of the embedding).

Our first main theorem (Theorem \ref{charvarquiverequal}) shows that all moduli
spaces of $G$-valued quiver representations are in fact isomorphic as varieties
(although not equal as $G$-spaces) to character varieties.  This result 
relies on an operation on quivers; namely, collapsing arrows.

Our second main theorem (Theorem \ref{pinchtheorem}) implies that for any free product of groups $\Gamma=\Gamma_1*\cdots*\Gamma_m$ if $\hom(\Gamma_i,K_\C)$ $K$-equivariantly strongly deformation retracts to $\hom(\Gamma_i,K)$ then the $K$-character variety of $\Gamma$ is a strong deformation retract of the $K_\C$-character variety of $\Gamma$, where $K_\C$ is the complexification of a compact Lie group $K$.  Theorem \ref{pinchtheorem} also implies that the moduli space of $K_\C$-representations of any quiver, denoted by $\mathcal{M}_{K_\C}$, strongly deformation retracts to the corresponding moduli space $\mathcal{M}_K$ (Theorem \ref{defretractthm}).  In fact, when certain relation types are imposed on the quiver representations the result still holds.  These results directly generalize the work in \cite{FlLa} and rely on another operation on quivers; namely, pinching vertices.  Preceding these theorems however, we discuss some general theory of $K_\C$-spaces and how they related to $K$-spaces.  In 
particular, we establish general criteria for such spaces to be related as cellular complexes and also general criteria for such spaces to be homotopy equivalent.

In the last section, we establish necessary quiver theoretic conditions for the moduli spaces of $G$-representations of quivers and the usual (additive) moduli spaces of quiver representations to correspond. In particular, let $Q$ be a quiver with no ends (sources or sinks) with Betti number $b_{1}(Q)=r$, and suppose a complex reductive Lie group $G$ is a subgroup of $\SL(n,\C)$ (which is always possible). Then our third main theorem (Theorem \ref{embedcharvarintovectquiv}) shows that the $G$-character variety of a $r$-generated group is naturally a subvariety of the usual moduli of quiver representations with fixed dimension vector $n$.  Also, when $G=\GL(n,\C)$, we show that the analogous embedding has dense image.  

Lastly, we show how some of affine toric geometry fits inside this framework, when $G=\mathbb{C}^{*}$ and the action is altered by including weights.

\section{Moduli of $G$-valued Quiver Representations}
\subsection{Quivers and Representations}
As in archery, a quiver is a collection of arrows.  We make this precise with the following definition.

\begin{defn}
A \emph{quiver} $Q$ is a finite directed graph. In other words, $Q$ is
formed by a finite set of \emph{vertices} $Q_V$, and a finite set of \emph{arrows}
$Q_A$ between vertices. For each arrow $a\in Q_A$, denote
by $h_{a},t_{a}\in Q_V$ its \emph{head} and \emph{tail} vertices.  Also let $N_A=\#Q_A$ and $N_V=\#Q_V$ be the 
cardinality of $Q_A$ and $Q_V$ respectively.
\end{defn}

A quiver is connected if it is path-connected as a 1-complex, i.e. if each edge corresponds to a 1-cell and each vertex a 0-cell and the 
graph theoretic definition defines the gluing maps.  We will mostly consider connected quivers as the general case does not
present many new features.  

Let $G$ be a group.

\begin{defn}
A $G$-\emph{valued representation} (or $G$-\emph{marking}) of $Q$ is a set mapping $f:Q_A\to G$.
\end{defn}

The $G$-valued representations form a set which is denoted by $\mathcal{F}_G(Q):=\fun(Q_A,G)$ or just
by $\mathcal{F}_G$ when the quiver is understood.  

Note that there is a canonical identification by the evaluation map: $$
\mathcal{F}_G(Q)\cong G^{N_A}.$$

With respect to the vertices we similarly define the group $\gauge_G(Q):=\fun(Q_V,G)\cong G^{N_V}.$
Here, the group law is given by component-wise multiplication: $g_1g_2(v)=g_1(v)g_2(v)$, for all $v\in Q_V$, and $g_1,g_2\in \gauge_G(Q)$. 
This follows since the evaluation map $\mathrm{ev}:\fun(Q_V,G)\to G^{N_V}$ is a group homomorphism.

The group $\mathcal{G}_G$, here after called the {\it gauge group} of $Q$, acts naturally on the set $\mathcal{F}_G$.
This action is defined by the following rule:\begin{equation}
\left(g\cdot f\right)(a):=g(h_{a})f(a)g(t_{a})^{-1},\quad\quad g\in\mathcal{G}_G,\ f\in\mathcal{F}_G.\label{eq:action}\end{equation}

The action is well defined since:
\begin{enumerate}
\item Clearly $\left(g\cdot f\right)(a)=g(h_{a})f(a)g(t_{a})^{-1}$ results in a new function in $\mathcal{F}_G$.
\item The identity in $\mathcal{G}_G$ is $\mathbf{I}=(I,...,I)\in G^{N_V}$ where $I$ is the identity in $G$.  So 
$\left(\mathbf{I}\cdot f\right)(a)=If(a)I^{-1}=f(a)$.
\item Lastly, \begin{eqnarray*}\left((g_1g_2)\cdot f\right)(a)&:=&(g_1g_2)(h_{a})f(a)(g_1g_2)(t_{a})^{-1}\\
                                                              &=&g_1(h_a)g_2(h_a)f(a)(g_1(t_a)g_2(t_a))^{-1}\\
                                                              &=&g_1(h_a)\left(g_2(h_a)f(a)g_2(t_a)^{-1}\right)g_1(t_a)^{-1}\\
                                                              &=&g_1\cdot\left(g_2\cdot f\right)(a).
              \end{eqnarray*}

\end{enumerate}

When $G$ is an algebraic group, the set $\mathcal{F}_G(Q)$ forms an affine algebraic variety which will be called the $G$-representation
variety of $Q$.  In this case, the action is algebraic since in each factor it arises from left multiplication of an algebraic group and right
multiplication composed with inversion in an algebraic group; all of which are given by polynomial mappings.

In order to consider moduli of quiver representations, let us first recall the theory of affine quotients by complex reductive groups.

\subsection{Complex affine reductive groups}

An algebraic group is a group that is an algebraic variety (zero set of a finite number of polynomials) such that the group operations are all regular (polynomial) mappings.  A complex affine group is an algebraic group that is the complex points of an affine variety.  Any affine group has a faithful linear representation (see \cite{Do} for instance), thus it is a closed subgroup of a general linear group and hence a linear Lie group.   Lie groups are smooth, and irreducible complex varieties are connected (see for instance \cite{Sh2} page 321).

Let $K$ be a compact Lie group.  Then $K$ is a real algebraic group which embeds in $\mathrm{O}(n,\mathbb{R})$ for some $n$.  Since $K$ is algebraic there is an ideal $\mathfrak{I}$ in the real coordinate ring $\mathbb{R}[\mathrm{O}(n,\mathbb{R})]$ defining its points.  Let $G=K_{\C}$ be the complex zeros of $\mathfrak{I}$, called the \emph{complexification} of $K$.  Then $G$ is a complex affine subgroup of $\mathrm{O}(n,\C)$ with coordinate ring $\C[G]=\mathbb{R}[K]\otimes_{\mathbb{R}}\C$.  Any complex affine group $G$ which arises in this fashion is called {\it reductive}.  The ``unitary trick'' shows $\SL(n,\C)$ is reductive.  We note that this definition, although not the most general, coincides with more general notions of reductivity when the algebraic group is complex linear.  In particular, another equivalent definition is that a complex linear algebraic group $G$ is reductive if for every finite dimensional representation of $G$ all subrepresentations have invariant complements.  The important 
observation is that such groups act like and have the algebraic structure of compact groups.  See \cite{Sch1}.

For example, $\mathrm{U}(n)=\{M\in\GL(n,\C)\ | \ M\overline{M}^{\mathsf{t}}=\id\}$, where $\id$ is the $n\times n$ identity matrix and $M^\mathsf{t}$ is the transpose of $M$.  Writing $M=A+\sqrt{-1}B$, we have that $\mathrm{U}(n)\cong$ $$\left\{\left(\begin{array}{cc}A&B\\-B&A\end{array}\right)\in \mathrm{GL}(2n,\mathbb{R})\ \ |\ A^\mathsf{t} A+B^\mathsf{t} B=I\ \&\ A^\mathsf{t} B- B^\mathsf{t} A=0\right\},$$  which sits isomorphically in $\mathrm{GL}(2n,\C)$ as $$\left\{\left(\begin{array}{cc}k&0\\0&(k^{-1})^\mathsf{t}\end{array}\right)\in \mathrm{GL}(2n,\C)\ |\ k\in \mathrm{U}(n)\right\}.$$  Letting $k$ be arbitrary in $\GL(n,\C)$ realizes the complexification $\mathrm{U}(n)_{\C}=\GL(n,\C)$.  In this way $\mathrm{U}(n)$ becomes the real locus of the complex variety $\GL(n,\C)$. Similarly, $\mathrm{SU}(n)_{\C}=\SL(n,\C)$.

\begin{rem}
We have not assumed $K$ is connected. Any compact Lie group $K$ has a finite number of connected components, 
all homeomorphic to the component containing the identity.   
As an algebraic variety, $\C[K_{\C}]$ has irreducible algebraic components (with respect to the Zariski topology).  
However, in this case the irreducible algebraic components are all disjoint homeomorphic topological components 
(with respect to the usual ball topology on $K_{\C}$), and each arises by complexifying a component of $K$ (see \cite{Bo} page 87). 
\end{rem}

\subsection{Algebraic Quotients}
A theorem of Nagata \cite{Na} says that if a reductive group acts on a finitely generated algebra $A$, then the subalgebra of invariants $A^G=\{a\in A\ |\ g\cdot a=a\}$ is likewise finitely generated.  This is one answer to Hilbert's fourteenth problem.

\begin{defn}
A {\it categorical quotient} of a variety $V_G$ with an algebraic group $G$ acting rationally is an object $V_G\quot G$ and a $G$-invariant morphism $\pi_G:V_G\to V_G\quot G$ such that the following commutative diagram exists 
uniquely for all invariant morphisms $f:V_G\to Z$: $$ \xymatrix{
V_G \ar[rr]^-{\pi} \ar[dr]_{f} & & V_G\quot G \ar@{.>}[dl]\\
& Z &} $$
It is a {\it good} categorical quotient if the following additionally hold:
\begin{enumerate}
\item[(i)] for open subsets $U\subset V_G\quot G$, $\C[U]\cong \C[\pi^{-1}(U)]^{G}$
\item[(ii)] $\pi$ maps closed invariant sets to closed sets
\item[(iii)] $\pi$ separates closed invariant sets.
\end{enumerate}
\end{defn}

When $G$ is reductive and $V_G$ is an affine $G$-variety, then $$V_G\to \mathrm{Spec}_{max}(\C[V_G]^G)$$ is a good categorical quotient.  See \cite{Do} for details.

It can be shown that the categorical quotient in the category of affine varieties (over $\C$ and with respect to a reductive group action) is also the categorical quotient for Hausdorff spaces or complex analytic varieties \cite{Lu2,Lu3}). 

Any such reductive quotient has an affine lift (see \cite{MFK}).  In other words, there is an affine space $\mathbb{A}^N$ for some potentially large $N$ where $V_G\subset \mathbb{A}^N$ and where the action of $G$ extends.  Then $$\Pi:\C[\mathbb{A}^N]\longrightarrow \C[V_G]$$ and more importantly $$\Pi_{G}:\C[\mathbb{A}^N\quot G]\longrightarrow \C[V_G\quot G]$$ are surjective morphisms.

In the case where $G$ is a complex reductive Lie group $K_\C$ arising as the complexification of a compact Lie group $K$, our main objects of interest are \[
\mathcal{M}_{K_\C}:=\mathcal{F}_{K_\C}\quot\mathcal{G}_{K_\C}\quad\mbox{and }\quad\mathcal{M}_K:=\mathcal{F}_K/\mathcal{G}_K.\]
Here, the first quotient is the affine GIT quotient, as the action of $\mathcal{G}_{K_\C}$ on $\mathcal{F}_{K_\C}$ is algebraic, whereas the second is the usual orbit space.  In this latter case, all orbits are compact, and in both cases the moduli space is Hausdorff.  Both of these two spaces are examples of the following definition.

\begin{defn}
Let $G$ be a topological group, and let $\mathcal{F}_{G}(Q)/\mathcal{G}_{G}(Q)$ denoted the orbit space corresponding to the action \eqref{eq:action}.  Then the identification space $\mathcal{M}_{G}(Q):=(\mathcal{F}_{G}(Q)/\mathcal{G}_{G}(Q))/\!\!\sim$ is called the {\it moduli space of $G$-representations of $Q$} (or {\it $G$-markings on $Q$}), where two $\mathcal{G}_{G}(Q)$-orbits are defined to be equivalent if and only if they are members of a chain of orbits whose closures pair-wise intersect.
\end{defn}

\begin{rem}\label{charvarremark}
Let $\rep_G(Q):=\hom(F_A,G)$ be the set of group homomorphisms from a free group $F_A$ of rank $N_A$ 
(freely generated by arrows) into $G$.  The evaluation mapping identifies $\rep_G(Q)$ and $\fnc_G(Q)$.  

In the case of a quiver with a single vertex the
gauge group $\mathcal{G}_{G}$ reduces to a single copy of $G$,
and the action becomes conjugation on $\rep_{G}(Q)$. Therefore,
in this case, the moduli spaces of $G$-markings on $Q$ and $G$-character
varieties of free groups are naturally isomorphic.

So with respect to $(G,X)$-spaces, the collection of pairs of the form 
$\{\gauge_G(Q),\fun(Q_A,G)\}$ properly contains the collection of pairs of the form
$\{(G,\hom(F,G))\},$ where $F$ is a free group.
\end{rem}

\subsection{Quivers with relations}

Let us consider paths inside a quiver $Q$. 
\begin{defn}
A \emph{path }of length $k\geq0$ is a sequence of arrows $a_{k}\cdots a_{1}$
such that the head of $a_{j}$ is the same vertex as the tail of $a_{j+1}$
for all $j=1,...,k-1$. Note that we are writing a path from right to
left. This way of writing is justified by Proposition \ref{pro:action-preserved}
below. A path of length one is a single arrow. We are including the
paths of length zero, in natural bijection with the vertices in $Q$.
We define the \emph{head} and \emph{tail} of a path $p=a_{k}\cdots a_{1}$
in the natural way: $t_{p}=t_{a_{1}}$, $h_{p}=h_{a_{k}}$. 
\end{defn}

\begin{figure}[h]
\begin{picture}(0,0)%
\includegraphics{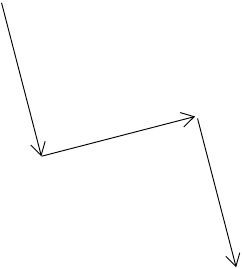}%
\end{picture}%
\setlength{\unitlength}{3947sp}%
\begingroup\makeatletter\ifx\SetFigFont\undefined%
\gdef\SetFigFont#1#2#3#4#5{%
  \reset@font\fontsize{#1}{#2pt}%
  \fontfamily{#3}\fontseries{#4}\fontshape{#5}%
  \selectfont}%
\fi\endgroup%
\begin{picture}(1930,2140)(4778,-2940)
\put(4974,-1175){\makebox(0,0)[lb]{\smash{{\SetFigFont{7}{8.4}{\familydefault}{\mddefault}{\updefault}{\color[rgb]{0,0,0}$a_1$}%
}}}}
\put(5565,-2085){\makebox(0,0)[lb]{\smash{{\SetFigFont{7}{8.4}{\familydefault}{\mddefault}{\updefault}{\color[rgb]{0,0,0}$a_2$}%
}}}}
\put(6658,-2221){\makebox(0,0)[lb]{\smash{{\SetFigFont{7}{8.4}{\familydefault}{\mddefault}{\updefault}{\color[rgb]{0,0,0}$a_3$}%
}}}}
\end{picture}%

\caption{Example Path $a_3a_2a_1$}
\end{figure}
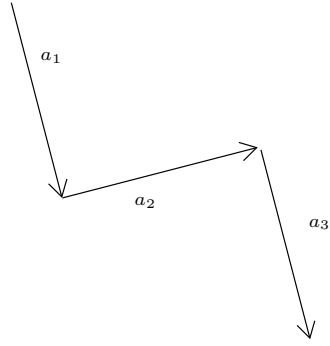

\begin{defn}
A \emph{quiver with relations} is a pair $(Q,R)$ where $Q$ is a
quiver and $R$ is a finite set of relations. A relation in a quiver
$Q$ is a (oriented) \emph{cycle} in $Q$, that is, a path in $Q$ with
the same head and tail. If $G$ is an arbitrary group, a representation
of $(Q,R)$ into $G$ is a function $f:Q_{A}\to G$
that satisfies all relations in the following way: for any cycle $p=a_{k}\cdots a_{1}\in R$,
$f(p):=f(a_{k})\cdots f(a_{1})=I$, where $I$ is the identity in $G$. 

\end{defn}

When $G$ is an algebraic group then
\begin{align*}
\mathcal{F}_G(Q,R) & \subset\mathcal{F}_G(Q),\end{align*}
is an algebraic subset and thus is a closed subset of $\mathcal{F}_G(Q)\cong G^{N_A}$.
\begin{prop}
\label{pro:action-preserved}The gauge group action of $\gauge_G(Q)$
on $\mathcal{F}_G(Q)$ preserves the subset $\mathcal{F}_G(Q,R)$.\end{prop}
\begin{proof}

We now verify that the action makes sense for quivers with relations.  In particular, given a cycle 
$a=a_{k}a_{k-1}\cdots a_{1}$ in $Q$ we let
the values of the representation $f\in \mathcal{F}_G(Q,R)$ of this cycle be denoted $f(a_{j})=A_{j}\in G$.  
Then ${\displaystyle 
\prod_{j=0}^{k-1} A_{k-j}=I}$  
and acting on it by $g\in\mathcal{G}_G$ we obtain ${\displaystyle \prod_{j=0}^{k-1} g(h(a_{k-j}))A_{k-j}
g(t(a_{k-j}))^{-1}.}$ But 
since $a$ is a cycle, (1)
$h(a_{j})=t(a_{j+1})$ for $1\leq j \leq k-1$ and (2) $t(a_{1})=h(a_{k})$.  Therefore (1) implies that 
$$\prod_{j=0}^{k-1} g(h(a_{k-j}))A_{k-j}g(t(a_{k-j}))^{-1}=g(h(a_{k}))\prod_{j=0}^{k-1} A_{k-j}g(t(a_{1}))^{-1}=
g(h(a_{k}))g(t(a_{1}))^{-1}$$ 
and (2) then implies $\prod g(h(a_{j})A_{j}g(t(a_{j}))^{-1}=I$, which shows that the $\mathcal{G}_G$-action preserves 
$\mathcal{F}_G(Q,R)$.

\end{proof}
In this way, one sees that representations of quivers with relations
generalize representations of finitely presented groups (see Remark \ref{charactervarietywithrelationsremark}). Therefore,
one is led to the following definition. 
\begin{defn}
Given a quiver with relations $(Q,R)$ and a complex affine reductive
group $K_\C$ with a choice of maximal compact subgroup $K$, the moduli space of $K_\C$-representations of $(Q,R)$ is defined to be the GIT quotient
$\mathcal{M}_{K_\C}(Q,R)=\mathcal{F}_{K_C}(Q,R)\quot\mathcal{G}_{K_\C}(Q)$. Likewise, we define 
$\mathcal{M}_{K}(Q,R)=\mathcal{F}_{K}(Q,R)/\mathcal{G}_{K}(Q)$ as the usual orbit space.
\end{defn}

\begin{rem}\label{charactervarietywithrelationsremark}
A quiver with relations $(Q,R)$ such
that $Q$ has only one vertex corresponds naturally to a finitely
presented group $\Gamma$.  In the same fashion as in Remark \ref{charvarremark} we see 
that its representations in $G$ are in correspondence to the $G$-representation variety 
of $\Gamma$ and therefore the moduli space of $G$-representations of
$(Q,R)$ is, in this case, the same as the $G$-character variety
of $\Gamma$; namely, $\hom(\Gamma,G)\quot G$.  Conversely, for any finitely presented
$\Gamma$ there exists a 1-vertex quiver with relations so that $\hom(\Gamma,G)\quot G=
\fnc_G(Q,R)\quot \gauge_G$.
\end{rem}

\subsection{Examples}

In this section we compute examples of $G$-valued quiver moduli spaces.  All of the examples are elementary.

Let $Q$ be the one-arrow quiver, then the moduli space is $G/G^2$, and since $G^2$ acts transitively on $G$ there is 
only one orbit and so the moduli is a single point.

Generalizing the example of one arrow, we compute the moduli space of representations of a quiver with a tail 
(see Figure \ref{tailquiver}).

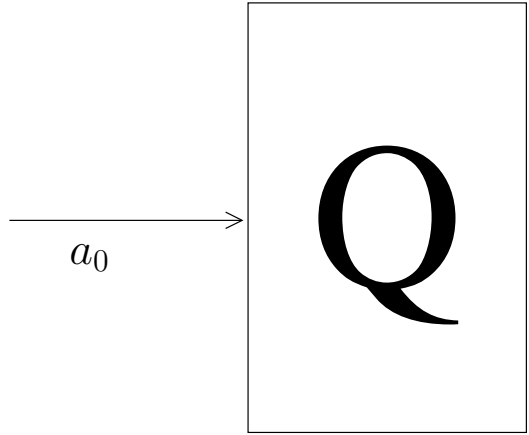
\begin{figure}[!!h]
\begin{picture}(0,0)%
\includegraphics{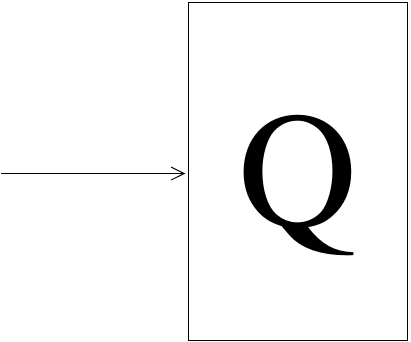}%
\end{picture}%
\setlength{\unitlength}{3947sp}%
\begingroup\makeatletter\ifx\SetFigFont\undefined%
\gdef\SetFigFont#1#2#3#4#5{%
  \reset@font\fontsize{#1}{#2pt}%
  \fontfamily{#3}\fontseries{#4}\fontshape{#5}%
  \selectfont}%
\fi\endgroup%
\begin{picture}(3271,2724)(7167,-3148)
\put(7556,-2077){\makebox(0,0)[lb]{\smash{{\SetFigFont{17}{20.4}{\familydefault}{\mddefault}{\updefault}{\color[rgb]{0,0,0}$a_0$}%
}}}}
\end{picture}%

\caption{A quiver with a tail.}\label{tailquiver}
\end{figure}

\begin{prop}
\label{prop:tailquiver} With respect to a quiver $Q$, let $\tilde{Q}$
be a quiver with a tail as in Figure $\ref{tailquiver}$. Then for
any group $G$, $\moduli_{G}(\tilde{Q})\cong\moduli_{G}(Q)$. \end{prop}
\begin{proof}
In this case the gauge group is $\gauge_{G}(\tilde{Q})=G\times\gauge_{G}(Q)$
and the $G$-quiver representations form $\fnc_{G}(\tilde{Q})=G\times\fnc_{G}(Q)$.

Let $F\in\fnc_{G}(Q)$ and $f=(f_{0},F)\in\fnc_{G}(\tilde{Q})$. Also
for any $h\in\gauge_{G}(Q)$ write $g=(g_{0},h)\in\gauge_{G}(\tilde{Q})=G\times\gauge_{G}(Q)$.
Let $N=N_{V}(Q)$ and label the vertices and arrows of $Q$ so that $t_{a_0}=v_0$ and $h_{a_0}=v_1$.  Then for any $h=(h_{1},...,h_{N})\in\gauge_{G}(Q)$,
\[
(h_{1}f_{0},h)\cdot f=((h_{1})f_{0}(h_{1}f_{0})^{-1},h\cdot F)=(I,h\cdot F).\]
 The elements that preserve this {}``normal form'' are parametrized
by $(h_{1},h_{1},...,h_{N_{V}})\in\gauge_{G}(\tilde{Q})$, and form
a subgroup isomorphic to $\gauge_{G}(Q)$.

Thus for any group $G$, $\moduli_{G}(\tilde{Q})$ is isomorphic to
the quotient of $\{I\}\times\fnc_{G}(Q)$ by $\gauge_{G}(Q)$ which
in turn is isomorphic to $\moduli_{G}(Q)$. The result follows. 
\end{proof}

By induction the above result implies that any arbitrarily long tail can be 
contracted on any quiver $Q$.  

In particular, we have
\begin{example}\label{prop:long-path}The moduli of $G$-representations of the long path 
(see Figure \ref{longpath}) is equivalent to the moduli space of a single arrow and thus both moduli spaces
are single points.  This coincides with the $G$-character variety of the trivial group $\hom(\langle 1\rangle,G)\quot G$.
\end{example}
There are many other similar examples, such as star-shaped quivers (the one point wedge of a union of long paths).

\begin{figure}[!!h]
\begin{picture}(0,0)%
\includegraphics{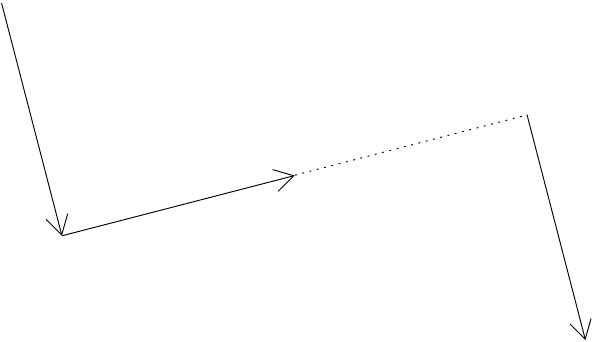}%
\end{picture}%
\setlength{\unitlength}{3947sp}%
\begingroup\makeatletter\ifx\SetFigFont\undefined%
\gdef\SetFigFont#1#2#3#4#5{%
  \reset@font\fontsize{#1}{#2pt}%
  \fontfamily{#3}\fontseries{#4}\fontshape{#5}%
  \selectfont}%
\fi\endgroup%
\begin{picture}(4740,2720)(7310,-2769)
\put(11868,-1964){\makebox(0,0)[lb]{\smash{{\SetFigFont{11}{13.2}{\familydefault}{\mddefault}{\updefault}{\color[rgb]{0,0,0}$a_m$}%
}}}}
\put(7601,-612){\makebox(0,0)[lb]{\smash{{\SetFigFont{11}{13.2}{\familydefault}{\mddefault}{\updefault}{\color[rgb]{0,0,0}$a_0$}%
}}}}
\put(8496,-1992){\makebox(0,0)[lb]{\smash{{\SetFigFont{11}{13.2}{\familydefault}{\mddefault}{\updefault}{\color[rgb]{0,0,0}$a_1$}%
}}}}
\end{picture}%

\caption{A long path Quiver}\label{longpath}
\end{figure}
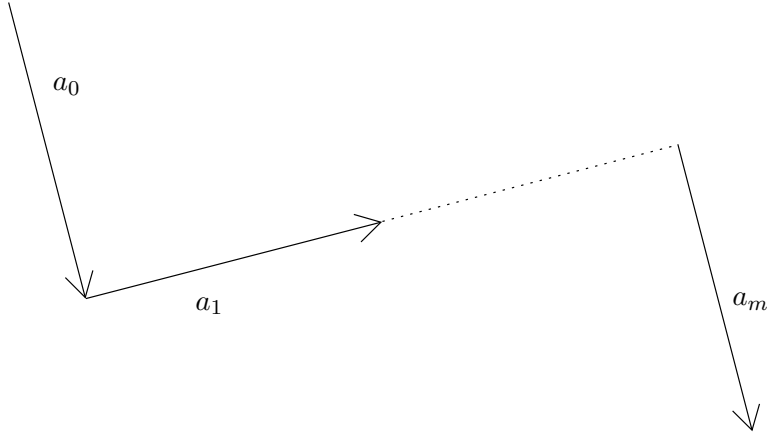

A second example is the long loop quiver. Let $Q$ be the quiver in Figure 
\ref{longloop}.

\begin{figure}[!!h]
\begin{picture}(0,0)%
\includegraphics{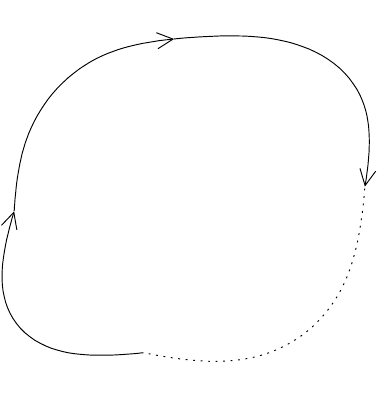}%
\end{picture}%
\setlength{\unitlength}{3947sp}%
\begingroup\makeatletter\ifx\SetFigFont\undefined%
\gdef\SetFigFont#1#2#3#4#5{%
  \reset@font\fontsize{#1}{#2pt}%
  \fontfamily{#3}\fontseries{#4}\fontshape{#5}%
  \selectfont}%
\fi\endgroup%
\begin{picture}(3018,3184)(3684,-4745)
\put(3749,-2102){\makebox(0,0)[lb]{\smash{{\SetFigFont{14}{16.8}{\familydefault}{\mddefault}{\updefault}{\color[rgb]{0,0,0}$a_0$}%
}}}}
\put(3847,-4662){\makebox(0,0)[lb]{\smash{{\SetFigFont{14}{16.8}{\familydefault}{\mddefault}{\updefault}{\color[rgb]{0,0,0}$a_m$}%
}}}}
\put(5521,-1758){\makebox(0,0)[lb]{\smash{{\SetFigFont{14}{16.8}{\familydefault}{\mddefault}{\updefault}{\color[rgb]{0,0,0}$a_1$}%
}}}}
\end{picture}%

\caption{A long loop Quiver}\label{longloop}
\end{figure}
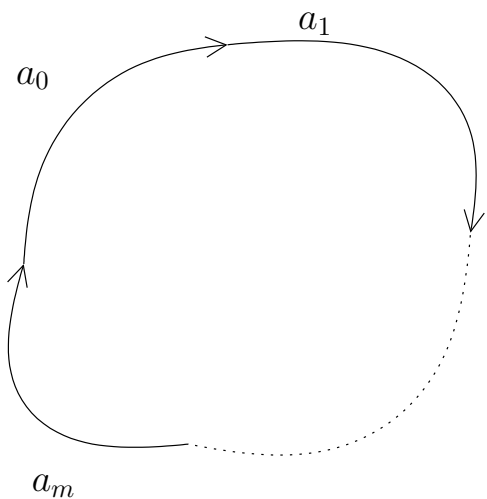

\begin{example}
The moduli of $G$-representations of the long loop 
$($see Figure $\ref{longloop})$ is equivalent to the moduli space of a single loop and 
thus is a $G$-character variety of a rank 1 free group. 
\end{example}

\begin{proof}
In this case,  $\mathcal{R}_G=G^m$ and $\mathcal{G}_G=G^{m}$ as well.  Then for any $f=(f_1,...,f_m)\in\mathcal{R}_G$ we 
consider 
$g=(I,f_1^{-1},(f_2f_1)^{-1},....,(f_{m-1}\cdots f_1)^{-1})$ in  $\mathcal{G}_G$. Then $g\cdot f$
\begin{eqnarray*}
&=&(f_1^{-1}f_1I,(f_2f_1)^{-1}f_2f_1,...,(f_{m-1}\cdots f_1)^{-1}f_{m-1}(f_{m-2}\cdots f_1),If_m(f_{m-1}\cdots f_1))\\
&=&(I,...,I,f_m\cdots f_1).
\end{eqnarray*}  
Clearly every such element of this form corresponds to a quiver representation, and the set of elements in $\gauge_G$ 
that preserve this form is the 
diagonal $\Delta=\{(g,...,g)\in \gauge_G\ |\ g\in G\}\subset \gauge_G$.  The action of 
$\Delta$ on $\{I\}^{m-1}\times G$ is conjugation in each factor (trivial in the first $m-1$ components).

Thus the moduli of a long loop is the same as the moduli of a single loop 
(1 vertex and 1 arrow).  We conclude that $\moduli_G(Q)=G\quot G$.  

\end{proof}

As shown in \cite{FlLa}, when $K$ is simply connected, $G\quot G$ and $K/K$ are both 
contractible.  In particular, $\moduli_K(Q)$ is homeomorphic to a closed 
real ball; and when $G=\mathsf{SL}(n,\C)$ we have $\moduli_G(Q)=\C^{n-1}$ for any 
long loop $Q$.

Combining these two examples we conclude that the moduli space of any comet-shaped quiver
(the one point wedge of a long loop wedged with a union of
long paths) is isomorphic to the moduli of one loop (in this case a rank 1 free group 
character variety). 

We will see in the next section that the examples computed in this section are special cases of a general phenomena.

\section{Character Varieties and Collapsing Arrows}

The methods used in the last section to determine moduli spaces for some example families of quivers suggests the 
consideration of operations on quivers.  In this section quivers are not assumed connected, unless stated 
otherwise.  However, we do assume that each connected component has at least one arrow.

There are two basic operations on quivers:  identifying two vertices (called {\it pinching}) which changes the gauge group but preserves the 
representation space, and removing an edge (called {\it clipping}) which changes the representation space but preserves the gauge group.  Each has an 
inverse that will be called {\it cloning} and {\it bridging}, respectively.  Composing them (in either order) at the same arrow gives a map we call {\it collapsing}
(with inverse mapping called {\it expanding}).

We now explore the effect of these maps on the moduli spaces of $G$-valued quiver representations and show they give results in the theory of 
character varieties.

\begin{figure}[h!!]
\begin{picture}(0,0)%
\includegraphics{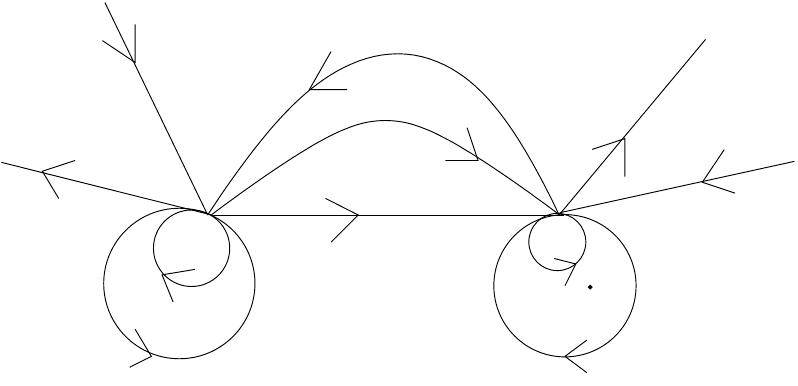}%
\end{picture}%
\setlength{\unitlength}{3947sp}%
\begingroup\makeatletter\ifx\SetFigFont\undefined%
\gdef\SetFigFont#1#2#3#4#5{%
  \reset@font\fontsize{#1}{#2pt}%
  \fontfamily{#3}\fontseries{#4}\fontshape{#5}%
  \selectfont}%
\fi\endgroup%
\begin{picture}(6367,3075)(4624,-2565)
\put(4704,-1244){\makebox(0,0)[lb]{\smash{{\SetFigFont{8}{9.6}{\familydefault}{\mddefault}{\updefault}{\color[rgb]{0,0,0}$\mathbf{a}_-$}%
}}}}
\put(5791,-156){\makebox(0,0)[lb]{\smash{{\SetFigFont{8}{9.6}{\familydefault}{\mddefault}{\updefault}{\color[rgb]{0,0,0}$\mathbf{a}_+$}%
}}}}
\put(5410,-520){\makebox(0,0)[lb]{\smash{{\SetFigFont{20}{24.0}{\familydefault}{\mddefault}{\updefault}{\color[rgb]{0,0,0}.}%
}}}}
\put(6139,-1374){\makebox(0,0)[lb]{\smash{{\SetFigFont{8}{9.6}{\familydefault}{\mddefault}{\updefault}{\color[rgb]{0,0,0}$t_{a_0}$}%
}}}}
\put(5522,-380){\makebox(0,0)[lb]{\smash{{\SetFigFont{20}{24.0}{\familydefault}{\mddefault}{\updefault}{\color[rgb]{0,0,0}.}%
}}}}
\put(6184,-2506){\makebox(0,0)[lb]{\smash{{\SetFigFont{8}{9.6}{\familydefault}{\mddefault}{\updefault}{\color[rgb]{0,0,0}$\mathbf{b}$}%
}}}}
\put(9798,-2114){\makebox(0,0)[lb]{\smash{{\SetFigFont{8}{9.6}{\familydefault}{\mddefault}{\updefault}{\color[rgb]{0,0,0}$\mathbf{c}$}%
}}}}
\put(9389,-1905){\makebox(0,0)[lb]{\smash{{\SetFigFont{20}{24.0}{\familydefault}{\mddefault}{\updefault}{\color[rgb]{0,0,0}.}%
}}}}
\put(7577,-373){\makebox(0,0)[lb]{\smash{{\SetFigFont{8}{9.6}{\familydefault}{\mddefault}{\updefault}{\color[rgb]{0,0,0}$\mathbf{f}_+$}%
}}}}
\put(7577,151){\makebox(0,0)[lb]{\smash{{\SetFigFont{8}{9.6}{\familydefault}{\mddefault}{\updefault}{\color[rgb]{0,0,0}$\mathbf{f}_-$}%
}}}}
\put(7620,-1026){\makebox(0,0)[lb]{\smash{{\SetFigFont{20}{24.0}{\familydefault}{\mddefault}{\updefault}{\color[rgb]{0,0,0}.}%
}}}}
\put(7615,-873){\makebox(0,0)[lb]{\smash{{\SetFigFont{20}{24.0}{\familydefault}{\mddefault}{\updefault}{\color[rgb]{0,0,0}.}%
}}}}
\put(7490,-1462){\makebox(0,0)[lb]{\smash{{\SetFigFont{8}{9.6}{\familydefault}{\mddefault}{\updefault}{\color[rgb]{0,0,0}$\mathbf{f}_0$}%
}}}}
\put(7615,-697){\makebox(0,0)[lb]{\smash{{\SetFigFont{20}{24.0}{\familydefault}{\mddefault}{\updefault}{\color[rgb]{0,0,0}.}%
}}}}
\put(9971,-1244){\makebox(0,0)[lb]{\smash{{\SetFigFont{8}{9.6}{\familydefault}{\mddefault}{\updefault}{\color[rgb]{0,0,0}$\mathbf{d}_-$}%
}}}}
\put(9447,-1963){\makebox(0,0)[lb]{\smash{{\SetFigFont{20}{24.0}{\familydefault}{\mddefault}{\updefault}{\color[rgb]{0,0,0}.}%
}}}}
\put(9058,-1418){\makebox(0,0)[lb]{\smash{{\SetFigFont{8}{9.6}{\familydefault}{\mddefault}{\updefault}{\color[rgb]{0,0,0}$h_{a_0}$}%
}}}}
\put(5896,-1976){\makebox(0,0)[lb]{\smash{{\SetFigFont{20}{24.0}{\familydefault}{\mddefault}{\updefault}{\color[rgb]{0,0,0}.}%
}}}}
\put(9915,-732){\makebox(0,0)[lb]{\smash{{\SetFigFont{20}{24.0}{\familydefault}{\mddefault}{\updefault}{\color[rgb]{0,0,0}.}%
}}}}
\put(5322,-653){\makebox(0,0)[lb]{\smash{{\SetFigFont{20}{24.0}{\familydefault}{\mddefault}{\updefault}{\color[rgb]{0,0,0}.}%
}}}}
\put(5873,-2083){\makebox(0,0)[lb]{\smash{{\SetFigFont{20}{24.0}{\familydefault}{\mddefault}{\updefault}{\color[rgb]{0,0,0}.}%
}}}}
\put(5896,-1888){\makebox(0,0)[lb]{\smash{{\SetFigFont{20}{24.0}{\familydefault}{\mddefault}{\updefault}{\color[rgb]{0,0,0}.}%
}}}}
\put(9954,-828){\makebox(0,0)[lb]{\smash{{\SetFigFont{20}{24.0}{\familydefault}{\mddefault}{\updefault}{\color[rgb]{0,0,0}.}%
}}}}
\put(9857,-615){\makebox(0,0)[lb]{\smash{{\SetFigFont{20}{24.0}{\familydefault}{\mddefault}{\updefault}{\color[rgb]{0,0,0}.}%
}}}}
\put(9971,-286){\makebox(0,0)[lb]{\smash{{\SetFigFont{8}{9.6}{\familydefault}{\mddefault}{\updefault}{\color[rgb]{0,0,0}$\mathbf{d}_+$}%
}}}}
\end{picture}%

\caption{A quiver neighborhood of an arrow $f_0$.}\label{fig:localarrow}
\end{figure}
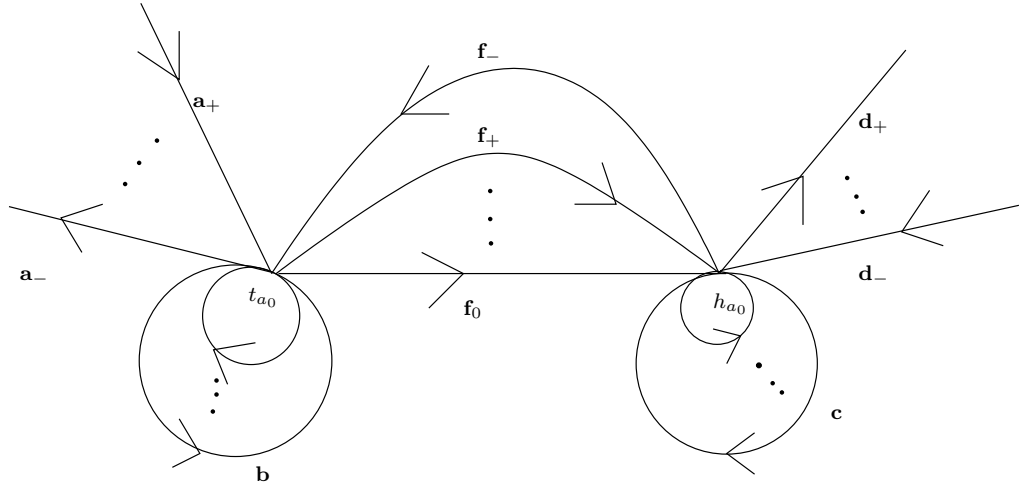

Let $Q=(Q_V,Q_A)$ be a connected quiver (with 2 or more arrows) and let $a_{0}\in Q_A$
be one of its arrows. Define $Q'$ as the quiver obtained from $Q$
by identifying the head and tail of $a_{0}$, and then removing $a_{0}$.
We will say that $Q'$ is obtained from $Q$ by {\it collapsing the arrow}
$a_{0}$. 

It is not difficult to see that the local picture of any
quiver around the particular arrow $a_{0}$ will be of the form indicated
in Figure \ref{fig:localarrow}. There will be arrows connecting $t_{a_{0}}$ to itself,
to $h_{a_{0}}$, and to other vertices of $Q$, and similarly for
$h_{a_{0}}$.

Let $G$ be a group. As in Figure \ref{fig:localarrow}, a quiver representation will
be labeled as a tuple $(\ab_+,\ab_-,\bb,f_{0},\fb_+,\fb_-,\cb,\db_+,\db_-,\mathbf{e})\in\mathcal{F}_{G}(Q)$ where
$f_{0}$ is associated with $a_{0}$, and any of the other letters
are associated with arrows in certain relative positions to the subquiver
$a_{0}\subset Q$. In particular, $\fb_+$ is the tuple of labels in $G$ associated to the arrows with tail
$t_{a_{0}}$ and head $h_{a_{0}}$ ($\fb_-$ has those tails and heads reversed), $\bb$ is the tuple of arrows whose head and tail is $t_{a_{0}}$, 
$\ab_+$ is the tuple associated with arrows that are not cycles and whose head is $t_{a_0}$ ($\ab_-$ is the tuple associated to arrows that are not cycles
and whose tail is $t_{a_0}$), $\mathbf{c}$ is 
the tuple associated with arrows whose head and tail is $h_{a_{0}}$, $\mathbf{d}_+$ is the tuple associated with 
arrows that are not cycles and whose tail is $h_{a_0}$ ($\db_-$ is the tuple associated to arrows that are not cycles and whose head is $h_{a_0}$), 
and $\mathbf{e}$ is the tuple associated with the arrows in $Q_A-L_{f_0}=\{e_1,...,e_{n_e}\}$, where $L_{f_0}$ is the union of all arrows local to $f_0$.  Precisely, $L_{f_0}=\{f_0,f_1,...,f_{n_f},
a_1,...,a_{n_a},b_1,...,b_{n_b},c_1,...c_{n_c},d_1,...,d_{n_d}\}$. If there are no arrows of some type (or combination of types) associated with 
$\mathbf{a},\mathbf{b},\mathbf{c},$ or $\mathbf{d}$, what follows can be easily adapted.

Let us write an element of $\mathcal{G}_{G}(Q)$ as $\gb=(g_{0},g_{1},...)\in G^{N_V}$,
where $g_{0}=g(t_{a_{0}})$ and $g_{1}=g(h_{a_{0}})$. Note that the
elements of $\mathcal{F}_{G}(Q')$ and $\mathcal{G}_{G}(Q')$ have
one less coordinate than those of $\mathcal{F}_{G}(Q)$ and $\mathcal{G}_{G}(Q)$;
we will write them in the form $(\mathbf{a}_+',\ab_-',\mathbf{b}',\mathbf{f}_+',\fb_-',\mathbf{c}',\mathbf{d}_+',\db_-',\mathbf{e}')\in\mathcal{F}_{G}(Q')$
and $(g_{1}',...)\in\mathcal{G}_{G}(Q')$. 

Let $\mathcal{M}_{Q}$ and $\mathcal{M}_{Q'}$ be the corresponding
moduli of $G$-valued representations.

Consider the map\begin{eqnarray*}
\mathsf{C}_{a_0}:\mathcal{M}_{Q} & \to & \mathcal{M}_{Q'}\\{}
[(\mathbf{a}_+,\ab_-,\mathbf{b},f_{0},\mathbf{f}_+,\fb_-,\mathbf{c},\mathbf{d}_+,\db_-,\mathbf{e})] & \mapsto & 
[(f_{0}\mathbf{a}_+,\ab_-f_0^{-1},f_{0}\mathbf{b}f_{0}^{-1},\mathbf{f}_+f_{0}^{-1},f_0\fb_-,\mathbf{c},\mathbf{d}_+,\db_-,\mathbf{e})]\end{eqnarray*}

We call the above mapping the collapsing map, and note that after performing a collapse of arrow $a_0$, the arrows $f_i$, $c_i$, and $b_i$ are all
one-arrow cycles at the same vertex.

\begin{prop}\label{prop:collapsingmap}
The map $\mathsf{C}_{a_0}$ is well-defined, and defines an isomorphism of algebraic
varieties.\end{prop}
\begin{proof}
We prove this for the moduli space of quiver representations without relations in detail, and then say a few words why the result holds true in general.

Let us first show that $\mathsf{C}_{a_0}$ is well-defined.  Let

\begin{eqnarray*}
\widetilde{\mathsf{C}_{a_0}}:\mathcal{F}_{Q} & \to & \mathcal{F}_{Q'}\\{}
(\mathbf{a}_+,\ab_-,\mathbf{b},f_{0},\mathbf{f}_+,\fb_-,\mathbf{c},\mathbf{d}_+,\db_-,\mathbf{e}) & \mapsto & 
(f_{0}\mathbf{a}_+,\ab_-f_0^{-1},f_{0}\mathbf{b}f_{0}^{-1},\mathbf{f}_+f_{0}^{-1},f_0\fb_-,\mathbf{c},\mathbf{d}_+,\db_-,\mathbf{e})\end{eqnarray*}
be the map on representations associated to $\mathsf{C}_{a_0}$.  To prove the latter is well-defined, it suffices to prove the former is gauge group
equivariant.

Let $\mathbf{h}_{0}$ (respectively $\mathbf{h}_{1}$) be the tuple of coordinates in $\mathbf{g}\in G^{N_V}$ associated to
the opposite end of an arrow that is not a cycle and whose intersection with $a_{0}$ is $t_{a_{0}}$ (respectively $h_{a_{0}}$); i.e. the opposite ends 
of the arrows associated to $\mathbf{a}$ (respectively $\mathbf{d}$).  

Some words on notation:  when we write $g\cdot f$ we mean the gauge group action on the quiver and if $f$
is a representation of a subquiver we mean the restricted action.  When we write $x\mathbf{v}$ we mean the element $x\in G$ is multiplied by each 
component of the tuple $\mathbf{v}$, $\mathbf{v}\mathbf{w}$ means component-wise multiplication of tuples, and $\mathbf{v}^{-1}$ means component-wise 
inversion.  Lastly, we will use $[\fb]$ to denote the orbit of a representation $\fb$.

The action of $\mathcal{G}_{G}(Q)$ on $\mathcal{F}_{G}(Q)$
is given by: 
\begin{eqnarray}\label{eq:collapse-action}
&\gb\cdot(\ab_+,\ab_-,\bb,f_{0},\fb_+,\fb_-,\cb,\db_+,\db_-,\eb)=\\ 
&(g_{0}\ab_+\hb_{0}^{-1},\hb_0\ab_-g_0^{-1},g_{0}\bb g_{0}^{-1},g_{1}f_{0}g_{0}^{-1},g_{1}\fb_+g_{0}^{-1},g_0\fb_-g^{-1}_1,
g_{1}\cb g_{1}^{-1},\hb_{1}\db_+g_{1}^{-1},g_1\db_-\hb_1^{-1},\gb\cdot\eb)\nonumber
\end{eqnarray}
so \begin{eqnarray*}
& &\widetilde{\mathsf{C}_{a_0}}(\gb\cdot(\ab_+,\ab_-,\bb,f_{0},\fb_+,\fb_-,\cb,\db_+,\db_-,\eb))    \\
&=&(g_{1}f_{0}g_{0}^{-1}g_{0}\ab_+\hb_{0}^{-1},\hb_0\ab_-g_0^{-1}(g_{1}f_{0}g_{0}^{-1})^{-1},g_{1}f_{0}g_{0}^{-1}g_{0}\bb g_{0}^{-1}
(g_{1}f_{0}g_{0}^{-1})^{-1},\\ 
& &g_{1}\fb_+g_{0}^{-1}(g_{1}f_{0}g_{0}^{-1})^{-1},g_{1}f_{0}g_{0}^{-1}g_0\fb_-g^{-1}_1,
g_{1}\cb g_{1}^{-1},\hb_{1}\db_+g_{1}^{-1},g_1\db_-\hb_1^{-1},\gb\cdot\eb)\\
&=&(g_{1}f_{0}\ab_+\hb_{0}^{-1},\hb_0\ab_-f_{0}^{-1}g_1^{-1},g_{1}f_{0}\bb f_{0}^{-1}g_{1}^{-1},g_{1}\fb_+f_{0}g_{1}^{-1},\\
& & g_{1}f_{0}\fb_-g^{-1}_1,g_{1}\cb g_{1}^{-1},\hb_{1}\db_+g_{1}^{-1},g_1\db_-\hb_1^{-1},\gb\cdot\eb)\\
& =&  (g_1,...)\cdot(f_{0}\ab_+,\ab_-f_{0}^{-1},f_{0}\bb f_{0}^{-1},\fb_+f_{0}^{-1},f_{0}\fb_-,\cb ,\db_+,\db_-,\eb)]\\
 & =& \gb'\cdot\widetilde{\mathsf{C}_{a_0}}((\ab_+,\ab_-,\bb,f_{0},\fb_+,\fb_-,\cb,\db_+,\db_-,\eb)),\end{eqnarray*} as was to be shown.
This finishes the proof that $\mathsf{C}_{a_0}$ is well defined, since equivariance implies the induced orbit space mapping is well
defined which in turn implies the induced affine GIT quotient mapping (where relevant) is well-defined since the latter is uniquely determined by the 
former.

Now, we show that $\mathsf{C}_{a_0}$ is bijective. Surjectivity is
clear, since we can just take $f_{0}=e$ to obtain any class $[(\ab_+,\ab_-,\bb,\fb_+,\fb_-,\cb,\db_+,\db_-,\eb)]\in\mathcal{M}_{Q'}$
as an image point of $\mathsf{C}_{a_0}$. 

To show injectivity, we prove that the orbit space mapping is injective.  This will imply that the corresponding GIT quotient mapping is also
injective since the affine GIT quotient is uniquely determined by its orbit space (it is a universal object).  Note that any point
$(\ab_+,\ab_-,\bb,f_0,\fb_+,\fb_-,\cb,\db_+,\db_-,\eb)\in\mathcal{F}_{G}(Q)$ is in the orbit of the point
$$(f_{0}\ab_+\hb_{0}^{-1},\hb_0\ab_-f_0^{-1},f_{0}\bb f_{0}^{-1},e,\fb_+f_{0}^{-1},f_0\fb_-,
c,\hb_{1}\db_+,\db_-\hb_1^{-1},\gb\cdot\eb)$$
upon acting by $\gb_{f_0}:=(g_{0},g_{1},...)=(f_{0},e,...)\in\mathcal{G}_{G}(Q)$.
Now, to preserve this form under the same action (Equation \ref{eq:collapse-action}),
we need to have $g_{0}=g_{1}\in G$. 

Now suppose that the orbits of $\widetilde{\mathsf{C}_{a_0}}(\rho_1)$ and $\widetilde{\mathsf{C}_{a_0}}(\rho_1)$ are equal for
$\rho_1,\rho_2\in\mathcal{F}_{G}(Q)$ in the ``normal form'' above. Then there exists $\gb'=(g_{1},...)\in\mathcal{G}_{G}(Q')$ so that 
$\gb'\cdot\widetilde{\mathsf{C}_{a_0}}(\rho_1)=\widetilde{\mathsf{C}_{a_0}}(\rho_2)$.  Let $\rho_1=
(\ab^1_+,\ab^1_-,\bb^1,e,\fb^1_+,\fb^1_-,\cb^1,\db^1_+,\db^1_-,\eb^1)$ and $\rho_2=
(\ab^2_+,\ab^2_-,\bb^2,e,\fb^2_+,\fb^2_-,\cb^2,\db^2_+,\db^2_-,\eb^2).$  Then 

\begin{eqnarray*}
& & (g_{1}\ab^1_+\hb_{0}^{-1},\hb_0\ab^1_-g_1^{-1},g_{1}\bb^1g_{1}^{-1},g_{1}\fb^1_+g_{1}^{-1},
g_{1}\fb^1_-g^{-1}_1,g_{1}\cb^1 g_{1}^{-1},\hb_{1}\db^1_+g_{1}^{-1},g_1\db^1_-\hb_1^{-1},\gb\cdot\eb^1)\\
&=&(\ab^2_+,\ab^2_-,\bb^2,\fb^2_+,\fb^2_-,\cb^2,\db^2_+,\db^2_-,\eb^2)
\end{eqnarray*}
and thus
\begin{eqnarray*}
& & (g_{1}\ab^1_+\hb_{0}^{-1},\hb_0\ab^1_-g_1^{-1},g_{1}\bb^1g_{1}^{-1},e,g_{1}\fb^1_+g_{1}^{-1},
g_{1}\fb^1_-g^{-1}_1,g_{1}\cb^1 g_{1}^{-1},\hb_{1}\db^1_+g_{1}^{-1},g_1\db^1_-\hb_1^{-1},\gb\cdot\eb^1)\\
&=&(\ab^2_+,\ab^2_-,\bb^2,e,\fb^2_+,\fb^2_-,\cb^2,\db^2_+,\db^2_-,\eb^2)
\end{eqnarray*}
and so
$(g_1,g_1,...)\cdot \rho_1=\rho_2$, precisely implying the orbit map is injective.

Since $\mathsf{C}_{a_0}$ and its set theoretic inverse are both given by polynomials (multiplication and inversion in an algebraic group are so given),
this defines an isomorphism of algebraic varieties (which implies it is a homeomorphism).

So we have shown that $\mathsf{C}_{a_0}$ determines $\moduli_Q\cong\moduli_{Q'}$ for a quiver $Q$ where $a_0$ is an arrow with different head
and tail.  It remains to prove that if we add in relations, then the mapping remains an isomorphism.

When considering quivers with relations, the cycle
$p d f_{0} a$ (read: first go along an arrow $a$ from the collection associated to $\ab_+$, then $f_{0}$, then an arrow $d$ associated to the $\db_+$, 
and lastly go along the path $p$ which ends where $a$ begins) goes precisely to $(p)(d)(f_{0}a)$ under $\mathsf{C}_{a_0}$.  This implies that cycle relations (with respect to their labelings) are preserved (clearly collapsing an arrow in a cycle of the quiver 
$Q$ results in another cycle in $Q'$).  So the mapping is well-defined between quivers with relations as well.   Since the moduli of quiver representations
is a subvariety of one without relations and the mapping above is globally injective, the restricted mapping is injective too.  Lastly, the above argument
for surjectivity still applies to the context of quivers with relations.  Thus, $\mathsf{C}_{a_0}$ is likewise an isomorphism between appropriate
moduli spaces of quiver representations with relations.

\end{proof}

\begin{rem}
We proved above that $\mathsf{C}_{a_0}$ is an isomorphism of GIT quotients when $G$ is a reductive algebraic group.  However, the above proof shows that $\mathsf{C}_{a_0}$ is a well-defined bijection of orbit spaces for any group $G$.  For example, we have an isomorphism of compact quotients when $G$ is compact. 
\end{rem}

This proposition allows us to generalize Example \ref{prop:long-path} as follows.

\begin{prop}
Let $Q$ be any connected quiver whose underlying graph is a tree. Then $\mathcal{M}_{G}(Q)$
is a single point.
\end{prop}

Any moduli space of quiver representations with one vertex is a character variety.  So character varieties are a subset of the 
moduli spaces of $G$-valued quivers; as sets, topological spaces, and varieties.  We now prove a surprising converse.  Namely, any moduli space
of $G$-valued quivers is algebraically (and consequently topologically) equivalent to a character variety despite that fact that as $(G,X)$-spaces the
representation of quivers (with relations) and the representations of finitely generated groups are not equivalent.

As mentioned earlier any quiver $Q$ has an associated one dimensional $CW$ complex,
and its first homology group (coefficients in $\mathbb{Z}$) is a free Abelian
group of a certain rank. Let us write $b_{1}(Q)=\mathsf{rk}\left(H_{1}(Q,\mathbb{Z})\right)$
and call this invariant the first Betti number of $Q$. 

\begin{thm}[Equivalence Theorem]\label{charvarquiverequal}Let $Q$ be a connected quiver, that is not contractible as a $\mathrm{CW}$ complex.
Then any moduli space of $G$-valued quivers, with or without relations, is algebraically isomorphic to a $G$-character variety which embeds into
$\hom(F_r,G)\quot G$ where $r=b_1(Q)$ and $F_r$ is a free group of rank $r$.  Otherwise, the moduli space is a point.
\end{thm}

\begin{proof}
By induction we can use Proposition \ref{prop:collapsingmap} repeatedly until all arrows are cycles.  From a homological perspective, the 1-complex 
has a 1-cell contracted and so does not change the value of $r=b_1(Q)$. So if there are no relations, then the result follows since a connected 
quiver with one vertex corresponds to a free group and the action is that of simultaneous conjugation.

For quivers with relations, Proposition \ref{prop:collapsingmap} again applies inductively until there is only one vertex of $Q$.  Thus quivers with 
relations are mapped to representations of finitely generated groups and since each collapsing map was an isomorphism the composition of them is as well.  The number of generators for the finitely generated group $\Gamma$ is the numbers of arrows which corresponds to 
$r=b_1(Q)$.  Thus $\moduli_G(Q,R)\cong \hom(\Gamma,G)\quot G\subset \hom(F_r,G)\quot G.$

\end{proof}

\begin{cor}
Let $Q$ be a quiver and $\tilde{Q}$ be the quiver $Q$ with some of its directions reversed.  Then $\moduli_G(Q)\cong 
\moduli_G(\tilde{Q})$.
\end{cor}

\begin{proof}
After collapsing all non-cycle arrows to a disjoint union of wedges of cycles and single arrows, the directions become irrelevant.
\end{proof}

Let $\zeta(G)$ be the center of the group $G$ and $\chi(Q)=1-b_{1}(Q)$ be the Euler characteristic of the quiver $Q$.

\begin{cor}\label{cor:dimension} Let $Q$ be a connected quiver.  Then $\dim\left(\moduli_{G}(Q)\right)=\dim(\zeta(G))-\dim(G)\chi(Q)$. \end{cor}

\begin{proof}
The dimension of the character variety of
a free group of rank $r$ is computed as follows: $\dim\hom(F_{r},G)\quot G=r\dim G-(\dim G-\dim\zeta(G))=\dim\zeta(G)+(r-1)\dim G$.
So the result follows from Theorem \ref{charvarquiverequal}. 
\end{proof}

Performing operations on quivers gives more flexibility in the study of 
the topology of the moduli of $G$-valued quiver representations and is therefore a 
potentially powerful tool in the general study of the 
moduli of representations; i.e, character varieties.  We will demonstrate this in the 
next section. 

\section{Strong Deformation Retractions and Pinching Vertices}
In what follows and unless otherwise stated, the topology on the moduli spaces we consider is the subspace topology induced by a polynomial embedding into affine space.  As stated earlier, up to homeomorphism, this topology is independent of the choice of polynomial embedding.

\subsection{Kempf-Ness Sets and Strong Deformation Retractions}

In this section we discuss results of \cite{KN}, \cite{Sch1}, and \cite{Ne} that are relevant to our situation.  

Let $V_G$ be an affine $G$-variety, $V_G\quot G=\mathrm{Spec}_{max}\left(\C[V_G]^G\right)$.

We may assume $V_G$ is equivariantly embedded as a closed subvariety of a representation $G\to \mathrm{GL}(V)$.  Let $\langle\ ,\ \rangle$ be a $K$-invariant Hermitian form on $V$ with norm denoted by $\norm{\ }$.  

Define for any $v\in V$ the mapping $p_v:G\to \mathbb{R}$ by $g\mapsto \norm{g\cdot v}^2$.  It is shown in \cite{KN} that any critical point of $p_v$ is a point where $p_v$ attains its minimum value.  Moreover, the orbit $G\cdot v$ is closed and $v \not= 0$ if and only if $p_v$ attains a minimum value.  

Define $\mathcal{KN}\subset V_G\subset V$ to be the set of critical points $\{v\in V_G\subset V\ |\ (dp_v)_{\id}=0\}$, where $\id\in G$ is the identity.  This set is called the \emph{Kempf-Ness} set of $V_G$. 
Since the Hermitian norm is $K$-invariant, for any point in $\mathcal{KN}$, its entire $K$-orbit is also contained in $\mathcal{KN}$.

Recall a strong deformation retraction from a space $X$ to a subspace $A$ is a homotopy relative to $A$ from the identity on $X$ to a retraction map $r:X\to A$.  In more detail, 
there is a continuous family of continuous functions $\{\phi_t:X\to X\ |\ t\in [0,1]\}$ such that (1) $\phi_0$ is the identity on $X$, (2) for all $t\in [0,1]$ $\phi_t|_A$ is the identity on $A$, and 
(3) $\phi_1(X)\subset A$. We note that \cite{Hatcher} simply calls this a deformation retraction.

The following theorem is proved in \cite{Sch1} making reference to \cite{Ne}.

\begin{thm}[Schwarz-Neeman]\label{schwarzneeeman}
The composition $\mathcal{KN}\to V_G\to V_G\quot G$ is proper and induces a homeomorphism $\mathcal{KN}/K\to V_G\quot G$ where $V_G\quot G$ has the subspace topology induced from its equivariant affine embedding.  Moreover, there is a $K$-equivariant deformation retraction of $V_G$ to $\mathcal{KN}.$
\end{thm}

Recall that all semi-algebraic sets (this includes varieties for obvious reasons) are cellular.  Precisely, we have the following theorem which can be found in \cite{BCR} on page 214.

\begin{thm}[Bochnak-Coste-Roy]\label{cellular}
Let $S$ be a closed and bounded semi-algebraic set.  Then given any finite family $\{S_i\}$ of semi-algebraic subsets of $S$, there exists a cellular decomposition of $S$ such that each $S_i$ is a sub-complex.
\end{thm}

We use this to establish that $W_K/K$ is generally a cellular sub-complex of the GIT quotient $V_G\quot G$.

\begin{thm}\label{cellularinclusion:general}
Let $G$ be a complex reductive group, and let $K$ be a maximal compact subgroup.  Let $V_G$ be a complex affine $G$-variety, and let $W_K\subset V_G$ be its real points. Assume further that $W_K$ is $K$-stable and compact. Then $W_K/K$ is canonically included in $V_G\quot G$ as a cellular sub-complex.
\end{thm}

\begin{proof}
Since $W_K$ is a the set of real points of $V_G$ there exists a set of generators and relations for the coordinate ring of $W_K$ so that the complex zeros of those relations equals $V_G$.
This implies $$\C[V_G]=\R[W_K]\otimes \C\supset \R[W_K]\otimes \R=\R[W_K].$$  We thus have 
$$\C[V_G]^{G}=\C[V_G]^{K}\supset \R[W_K]^{K},$$ where the first equality follows from the 
``unitary trick'' (see \cite{Do}, or \cite{Sch3}).  Note that $$\R[W_K]^{
K}\otimes \C=(\R[W_K]\otimes \C)^{
K}=\C[V_G]^{K}=\C[V_G]^{G}.$$  Therefore
there exists a generating set $\{p_1,...,p_N\}$ for $\R[W_K]^{K}$ which is also a generating set for $\C[V_G]^{G}$.
Let $P=(p_1,...,p_N)$ be the corresponding polynomial mapping, and note the generators of each invariant ring determine closed orbits. 
Then $P(V_G)=V_G\quot G$, and also $P(W_K)=W_K/K$ (see \cite{Sch1}).  This in turn 
implies that $W_K/K=P(W_K)\subset P(V_G)= V_G\quot G$.

By \cite{PS} we know $W_K/K$ is semi-algebraic.  Since we have assumed that $W_K$ and $K$ are compact, we know $W_K/K$ is compact and thus closed.  We just showed that $W_K/K$ is canonically included in $V_G\quot G$ by a polynomial mapping.  Thus, $W_K/K$ may be considered as a semi-algebraic subset of $V_G\quot G$ (it is semi-algebraic since it is a variety).  Therefore, Theorem \ref{cellular} implies that $W_K/K$ is a sub-complex of $V_G\quot G$, as required.

\end{proof}

We now are in a position to prove an important tool in analyzing the topology of moduli spaces.

\begin{thm}\label{kempnesscorollary}
Let $G$ be a complex reductive group, and let $K$ be a maximal compact subgroup.  Let $V_G$ be a complex affine $G$-variety, and let $W_K\subset V_G$ be its real points. Assume further that $W_K$ is $K$-stable and compact.  If $W_K$ is a subset of the Kempf-Ness set, and there exists a $K$-equivariant strong deformation retraction of $V_G$ onto $W_K$, then $V_G\quot G$ strongly deformation retracts onto $W_K/K$. 
\end{thm}

\begin{proof}
Theorem \ref{schwarzneeeman} implies $V_G/K$ deformation retracts to $\mathcal{KN}/K\cong V_G\quot G$.  Thus $V_G\quot G$ is homotopy equivalent to $V_G/K$.  On the other hand, our hypothesis implies $V_G/K$ deformation retracts onto $W_K/K$; again they are homotopy equivalent.  Thus, $V_G\quot G$ is homotopy equivalent to $W_K/K$ via the canonical inclusion mapping from Theorem \ref{cellularinclusion:general}.  In particular, this implies, since $W_K\subset \mathcal{KN}$, that the inclusion mapping induces isomorphisms $\pi_m(V_G\quot G)\cong \pi_m(W_K/K)$ for all $m\geq 0$.

However, Theorem \ref{cellularinclusion:general} tells that $W_K/K$ is a sub-complex of the cellular complex $V_G\quot G$.  Thus, Whitehead's Theorem (see \cite{Hatcher}, page 346) implies that $V_G\quot G$ strongly deformation retracts onto $W_K/K$.

\end{proof}

\subsection{Cellular Inclusion Theorem}
Let $K$ is a compact Lie group and $G=K_\C$.  We may assume $K\subset \mathrm{O}(n,\R)$ is a real affine variety, by the Peter-Weyl theorem;
and thus $G\subset \mathrm{O}(n,\C)$.  Our goal is to now prove that $\mathcal{F}_G(Q)\quot\mathcal{G}_{G}$ deformation retracts onto
$\mathcal{F}_K(Q)/\mathcal{G}_{K}$; we first prove it is canonically embedded.

\begin{prop}\label{kmodsubsetgmod}
Let $Q$ be any quiver, and let $R$ be a set of relations.
Then $\fnc_K(Q,R)/\mathcal{G}_K$ is canonically included as a cellular sub-complex of $\fnc_G(Q,R)\quot\mathcal{G}_G$.
\end{prop}

\begin{proof}
There exists a set of generators and relations for the coordinate ring of $K$ so that the complex zeros of those relations equals $G$ (see \cite{Sch1}).
This implies $\fnc_G(Q)=G^{N_A}=(K_\C^{N_A})=(K^{N_A})_\C=\fnc_K(Q)_\C,$ which in turn implies that $\fnc_G(Q,R)=\fnc_K(Q,R)_\C.$  Thus, we are in the setting of Theorem \ref{cellularinclusion:general}, which provides our desired result.
\end{proof}

Denote the inclusion mapping $\iota$ which, as the proof shows, is defined by sending a $\gauge_K$-orbit to the $\gauge_G$-orbit (uniquely) determined by
any representative from the $\gauge_K$-orbit.  As it is given by polynomials, it is clearly continuous.

\begin{cor}\label{injection}
Any two $K$-valued quiver representations that are $\mathcal{G}_G$-equivalent are also $\mathcal{G}_K$-equivalent.
\end{cor}

\begin{proof}
If not the mapping $\iota$ from Proposition \ref{kmodsubsetgmod} would not be injective. 
\end{proof}

\begin{rem}
The above corollary generalizes the analogous statement in \cite{FlLa} for character varieties (see Remark \ref{charvarremark}).  In that article we gave two proofs of the character variety version of Corollary \ref{injection}.  We now take a moment to correct the one in the appendix.  Note that the polar decomposition is unique (else it would not give a diffeomorphism).  If two $K$-valued representations, $\rho$ and $\phi$, of finitely generated groups are $G$-conjugate, then there exists a $g\in G$ such that for any word $w$, $k:=\phi(w)=g\rho(w) g^{-1}:=gk_1g^{-1}$. Then $k=gk_1g^{-1}=k_2e^pk_1e^{-p}k_2^{-1}$ which implies $$k_3:=k_2^{-1}kk_2=e^pk_1e^{-p}=k_1e^{ad_{k_1^{-1}}p}e^{-p}:=k_1k_4e^{p_1}.$$  Thus,
$k_4^{-1}k_1^{-1}k_3=e^{p_1}$ which implies $e^{p_1}=I$ which implies $e^{ad_{k_1^{-1}}p}=k_4e^p$ which implies $k_4=I$ and $e^{ad_{k_1^{-1}}p}=e^p$.  Therefore, $k_3=k_1$ which implies $k=k_2k_1k_2^{-1}$.  In conclusion, $G$-conjugate $K$-representations of finitely generated groups are $K$-equivalent.
\end{rem}

\subsection{Strong Deformation Retraction of Quiver Moduli}
All quivers we consider in this section are connected with at least one arrow, which implies, for instance, that each vertex is incident with at least one arrow.

We define an equivalence relation on quivers.  Two quivers are said to be arrow-equivalent if and only if there exists a a bijection between their sets
of arrows.  Clearly this is an equivalence relation since composition of bijections is transitive, invertible, and reflexive.

We define the pinching mapping on a quiver to be the mapping that takes two vertices and identifies them without otherwise changing the arrow.  
This operation preserves the representations spaces but kills one factor in the gauge group.  The resulting quiver is arrow-equivalent to the
first.

\begin{lem}\label{pinchlemma}
Let $K$ be a compact Lie group and let $K_\C$ be its complexification.  Let $(Q_1,R_1), (Q_2,R_2)$, and $(Q_3,R_3)$ be arrow-equivalent quivers with relations.  We make the following three additional assumptions:
\begin{enumerate}
\item[(a)] $\fnc_{K_\C}(Q_1,R_1)\subset \fnc_{K_\C}(Q_2,R_2)$,
\item[(b)] there exists a strong deformation
retraction $\Phi_t$ that is $\gauge_{K}(Q_3)$-equivariant from $\fnc_{K_\C}(Q_2,R_2)$ onto $\fnc_{K}(Q_2,R_2)$,
\item[(c)] for all time $t$, $\Phi_t(\fnc_{K_\C}(Q_1,R_1))\subset \fnc_{K_\C}(Q_1,R_1)$.  
\end{enumerate}
Then
\begin{enumerate}
\item $\fnc_{K_\C}(Q_1,R_1)$ and $\fnc_{K_\C}(Q_2,R_2)$ are $\gauge_{K_\C}(Q_3)$-stable,
\item $\Phi_t(\fnc_{K}(Q_1,R_1))=\fnc_K(Q_1,R_1)$ for all $t\in [0,1]$,
\item $\Phi_t|_{\fnc_{K_\C}(Q_1,R_1)}$ is a $\gauge_K(Q_3)$-equivariant strong deformation retraction onto $\fnc_{K}(Q_1,R_1)$.
\end{enumerate}

\end{lem}

\begin{proof}
Item (1): This result is simply a technical observation highlighting that the hypotheses of the theorem are sensible.  For any group, whenever $Q_1$ and $Q_2$ are arrow equivalent, $\fnc_G(Q_1)=\fnc_G(Q_2)$.  Consequently 
$\gauge_G(Q_i)$ acts on $\fnc_G(Q_j)$ for any choices of $i,j\in\{1,2,3\}$.   Since the gauge groups generally preserve relations, the result follows.

Item (2): $\fnc_K(Q_1,R_1)$ are the representations that satisfy the relations $R_1$ and have values in $K$.  Since 
$\fnc_{K_\C}(Q_1,R_1)\subset \fnc_{K_\C}(Q_2,R_2)$ we conclude that $\fnc_K(Q_1,R_1)\subset \fnc_K(Q_2,R_2)$.  However, $\Phi_t=\mathrm{id}$ on
$\fnc_K(Q_2,R_2)$, and therefore,  $\Phi_t(\fnc_{K}(Q_1,R_1))=\fnc_K(Q_1,R_1)$.

Item (3): $\gauge_K(Q_3)$-equivariance and continuity follows since the map is equivariant and continuous on all of $\fnc_{K_\C}(Q_2,R_2)$.  
Since $\Phi_0$ is the identity on $\fnc_{K_\C}(Q_2,R_2)$, it is the identity on the subset $\fnc_{K_\C}(Q_1,R_1)$.  The hypothesis 
$\Phi_t(\fnc_{K_\C}(Q_1,R_1))\subset \fnc_{K_\C}(Q_1,R_1)$ implies
we have a continuous family of self mappings for all $t$ that starts at the identity, and is $\gauge_K(Q_3)$-equivariant. Item 2. shows that $\Phi_t$ is
the identity on $\fnc_K(Q_1,R_1)$.  Lastly, note that $\Phi_1(\fnc_{K_\C}(Q_2,R_2))\subset \fnc_K(Q_2,R_2)$ by assumption.  Thus, 
$$\Phi_1(\fnc_{K_\C}(Q_1,R_1))\subset \Phi_1(\fnc_{K_\C}(Q_2,R_2))\subset \fnc_K(Q_2,R_2)$$ and so 
$\Phi_1(\fnc_{K_\C}(Q_1,R_1))\subset \fnc_K(Q_1,R_1)$ since satisfying the relations $R_1$ and having values in $K$ definitely implies membership in
$ \fnc_K(Q_1,R_1)$.
\end{proof}

\begin{thm}\label{pinchtheorem}

Let $(\mathcal{Q}_1,\mathcal{R}_1),...,(\mathcal{Q}_n,\mathcal{R}_n)$ be a collection of quivers with relations.  Define $\mathcal{R}=\mathcal{R}_1\cup\cdots\cup \mathcal{R}_n$, and $\mathcal{Q}$ to be the image of some finite collection of pinching maps applied to $\mathcal{Q}_1\cup\cdots \cup \mathcal{Q}_n$.  Suppose for each index $1\leq i\leq n$, $\fnc_G(\mathcal{Q}_i,\mathcal{R}_i)$ $\gauge_K(\mathcal{Q}_i)$-equivariantly strong deformation retracts to $\fnc_K(\mathcal{Q}_i,\mathcal{R}_i)$.  Then $\moduli_G(\mathcal{Q},\mathcal{R})$ strong deformation retracts onto $\moduli_K(\mathcal{Q},\mathcal{R})$.
\end{thm}

\begin{proof}

In the hypotheses of Theorem \ref{pinchtheorem}, let 
$(Q_2,R_2)$ be  $(\mathcal{Q}_1\cup\cdots \cup \mathcal{Q}_n,\mathcal{R})$, let $(Q_1,R_1)$ be $(\mathcal{Q},\mathcal{R})$, and let $Q_3=Q_1$ (note that $R_3$ never served a purpose so we can suppose $R_3=\emptyset$). Clearly, $Q_1,Q_2,Q_3$ are arrow equivalent.

$\fnc_{K_\C}(Q_1,R_1)\subset \fnc_{K_\C}(Q_2,R_2)$ since they are equal in this case since pinching does not change the set of 
cycles in the relations sets.

There exists a strong deformation retraction $\Phi_t$ that is $\gauge_{K}(Q_2)$-equivariant from $\fnc_{K_\C}(Q_2,R_2)$ onto 
$\fnc_{K}(Q_2,R_2)$ by the hypothesis of our theorem.  However, the pinching operation gives $\gauge_K(Q_3)=\gauge_K(Q_1)\subset \gauge_K(Q_2)$ is a subgroup (diagonally embed $K\hookrightarrow K\times K$ at identified vertices), and thus the strong deformation retraction is likewise $\gauge_{K}(Q_3)$-equivariant.

Lastly, for all time $t$, $\Phi_t(\fnc_{K_\C}(Q_1,R_1))\subset \fnc_{K_\C}(Q_1,R_1)$ since 
$\fnc_{K_\C}(Q_1,R_1)=\fnc_{K_\C}(Q_2,R_2)$ and the stability on the latter is assumed.

We have just proved that Theorem \ref{pinchtheorem} implies $\fnc_G(\mathcal{Q},\mathcal{R})$ $\gauge_K(\mathcal{Q})$-equivariantly strong deformation retracts to $\fnc_G(\mathcal{Q},\mathcal{R})$.

Thus Theorem \ref{kempnesscorollary} implies the theorem since Proposition \ref{kmodsubsetgmod} implies $\moduli_K(\mathcal{Q},\mathcal{R})$ is canonically included as a cellular sub-complex of $\moduli_G(\mathcal{Q},\mathcal{R})$, and it is in the Kempf-Ness set by Proposition \ref{kempfnesssection}.

\end{proof}

\begin{cor}\label{freeprodcor}
Let $\Gamma_1$,...,$\Gamma_m$ be finitely generated groups and let $\Gamma_1*\cdots*\Gamma_m$ be their free product. 
Let $K$ be a compact Lie group, and $K_\C$ be its complexification.  
If $\hom(\Gamma_i,K_\C)$ $K$-equivariantly strong deformation retracts to $\hom(\Gamma_i,K)$ for all $1\leq i\leq m$, then
$\hom(\Gamma_1*\cdots*\Gamma_m,K_\C)\quot K_\C$ strongly deformation retracts to $\hom(\Gamma_1*\cdots*\Gamma_m,K)/K$.
\end{cor}

\begin{proof}
For $1\leq m\leq m$, let $(\mathcal{Q}_i,\mathcal{R}_i)$ be one-vertex quivers each of whose relations corresponds to the finitely generated
groups $\Gamma_1,...,\Gamma_m$.  Then identifying (by {\it pinching}) the $m$-vertices in $\mathcal{Q}_1\cup\cdots\cup \mathcal{Q}_m$ 
together to obtain $Q$ we find ourselves in the context of Theorem \ref{pinchtheorem}.  So the result follows.
\end{proof}

\begin{rem}

This is a very special case of the Theorem \ref{pinchtheorem}.  However, notice that this immediately implies Theorem
4.3 in \cite{FlLa} since the case of ``one-loop'' quivers (1 vertex and 1 arrow) give a single copy of $G$ with $G$-action being 
conjugation.
\end{rem}

\begin{rem}
One can directly prove the above corollary by observing $\hom(\Gamma_1*\cdots*\Gamma_m,G)\cong \hom(\Gamma_1,G)\times\cdots \times \hom(\Gamma_m,G)$, 
and indeed, we originally did.  However, generalizing this result we came to Theorem \ref{pinchtheorem}.
\end{rem}

\begin{cor}[Strong Deformation Retraction Theorem]\label{defretractthm}
There is a strong deformation retraction of 
$\fnc_G(Q)\quot\mathcal{G}_G$ onto $\fnc_K(Q)/\mathcal{G}_K$.
\end{cor}

\begin{proof}
We again use Theorem \ref{pinchtheorem}.  Let $N_A$ be the number of arrows in $Q$, and let $\mathcal{Q}_i$ be $N_A$ 1-arrow quivers with no relations.  Then $Q$ is obtained from $\{\mathcal{Q}_i\}$ by a finite sequence of pinching maps.  To complete the proof we need only show that each 1-arrow quiver $G$-valued representation space $\gauge_K(Q)$-equivariantly strong deformation retracts to the corresponding $K$-valued space.

Consider the polar decomposition for the complex reductive Lie group
$G$ with maximal compact $K$. The multiplication map\[m:
K\times\exp\mathfrak{p}\to G\]
defines a diffeomorphism (see \cite{K}, page 384), where $\mathfrak{g}=\mathfrak{k}\oplus\mathfrak{p}$
is a Cartan decomposition of the Lie algebra of $G$, with $\mathfrak{k}$
being the Lie algebra of $K$. %

As we stated in \cite{FlLa}, the inverse to $m$ can be defined explicitly by\begin{eqnarray}
m^{-1}:G & \to & K\times\exp\mathfrak{p}\nonumber \\
g & \mapsto & \left(g(g^{*}g)^{-\frac{1}{2}},(g^{*}g)^{\frac{1}{2}}\right),\label{eq:m-1}\end{eqnarray}
where $g^{*}$ denotes the Cartan involution on $G$ applied to $g$
(which, in the case of $\SL(n,\C)$ and $\SU(n)$, is the usual conjugate
transpose map). The formula above follows from the fact that, if we
write $g=ke^{p}$, for $k\in K$ and $p\in\mathfrak{p}$, then $g^{*}=e^{p^{*}}k^{*}=e^{p}k^{*}$,
which implies $g^{*}g=e^{2p}$ (since the Cartan involution fixes
any element of $\mathfrak{p}$).

The family of maps, \begin{eqnarray}
\phi_{t}:G & \to & G\nonumber \\
g=ke^{p} & \mapsto & g(g^{*}g)^{-\frac{t}{2}}=ge^{-tp}=ke^{(1-t)p},\label{eq:phi}\end{eqnarray}
for $t\in[0,1]\subset\mathbb{R}$, provides a strong deformation retraction
from $G$ to $K$. Moreover, for any $t\in[0,1]$, $\phi_{t}$ is
$K$-equivariant, \[
\phi_{t}(k\cdot g)=k\cdot\phi_{t}(g),\]
with respect to the conjugation action of $K$ on $G$:\[
k\cdot g:=kgk^{-1},\quad\mbox{for }k\in K,g\in G.\]

We now return to the moduli space $\fnc_{G}(Q)\quot \gauge_G$ of $G$-representations
of a fixed quiver $Q$. We can naturally define $\Phi_{t}:\mathcal{F}_G(Q)\to\mathcal{F}_G(Q)$
as the map $\phi_{t}$ component-wise for all arrows $a\in Q_A$, that
is:\[
\left(\Phi_{t}\left(f\right)\right)(a):=\phi_{t}(f(a)),\quad\quad f\in\mathcal{F}_G(Q).\]

In general, if $r_1,...,r_m$ are strong deformation retractions of spaces $T_1,...,T_m$ onto subspaces $S_1,...,S_m$, then 
$$(r_1,...,r_m):T_1\times\cdots\times T_m\to T_1\times\cdots\times T_m$$ is a strong deformation retraction onto $S_1\times\cdots\times S_m$ with 
respect to the product topology.  Thus since $\Phi_t:\mathcal{F}_G(Q)\to\mathcal{F}_G(Q)$ is equivalent to $\phi^{N_A}_t:G^{N_A}\to G^{N_A}$, 
we conclude $\Phi_t$ is a strong deformation retract.  Note that the affine embedding topology of $\mathcal{F}_G(Q)$ corresponds to the the product 
topology on $G^{N_A}$ where $G$ is given the affine embedding topology.

To prove $\mathcal{G}_{K}$-equivariance we first note that for any real number $t\geq0$, we have\[
\left(he^{tp}h^{-1}\right)=\left(he^{p}h^{-1}\right)^{t}\]
for all $h\in K$ and $p\in\mathfrak{p}$. The formula certainly works
for rational $t$ (one can see this directly by writing down the Taylor series), and the general case follows by continuity.
Using this, we compute at each arrow $a\in Q_A$, the action under 
$\gamma\in\mathcal{G}_{K}$.  We use the notation $\gamma_{h_a}$ instead of $\gamma(h_a)$ to make this computation more readable.
\begin{eqnarray*}
\left(\Phi_{t}(\gamma\cdot f)\right)(a) & = & \phi_{t}(\gamma_{h_{a}}f(a)\gamma_{t_{a}}^{-1})\\
 & = & \gamma_{h_{a}}f(a)\gamma_{t_{a}}^{-1}\left((\gamma_{h_{a}}f(a)\gamma_{t_{a}}^{-1})^{*}\gamma_{h_{a}}f(a)\gamma_{t_{a}}^{-1}
\right)^{-\frac{t}{2}}\\
 & = & \gamma_{h_{a}}f(a)\gamma_{t_{a}}^{-1}\left(\gamma_{t_{a}}f(a)^{*}f(a)\gamma_{t_{a}}^{-1}\right)^{-\frac{t}{2}}\\
 & = & \gamma_{h_{a}}f(a)\gamma_{t_{a}}^{-1}\left(\gamma_{t_{a}}e^{2p_{a}}\gamma_{t_{a}}^{-1}\right)^{-\frac{t}{2}}\\
 & = & \gamma_{h_{a}}f(a)\gamma_{t_{a}}^{-1}\gamma_{t_{a}}e^{-tp_{a}}\gamma_{t_{a}}^{-1}\\
 & = & \gamma_{h_{a}}f(a)e^{-tp_{a}}\gamma_{t_{a}}^{-1}\\
 & = & \gamma_{h_{a}}\phi_{t}(f(a))\gamma_{t_{a}}^{-1}\\
 & = & \gamma\cdot\phi_{t}(f(a)),\end{eqnarray*}
where we used $f(a)=k_{a}e^{p_{a}}\in G$, $k_{a}\in K,\ p_{a}\in\mathfrak{p}$.
Note also that, strictly speaking, this shows that $\Phi_{t}$ is
equivariant for the action of $\mathcal{G}_{K}$ at a single arrow $a\in A$,
but since the action is factor-wise in $\mathcal{F}_G(Q)\cong G^{N_A}$, this calculation 
implies equivariance for each factor simultaneously; thus
$\Phi_{t}$ is $\mathcal{G}_{K}$-equivariant.
\end{proof}

\begin{rem}
In the appendix we determine the Kempf-Ness set that realizes the isomorphism $\moduli_G(Q)\cong \mathcal{KN}/\gauge_K(Q)\supset \moduli_K(Q)$.
\end{rem}

\section{Super-Cyclic Quivers and Additive Quiver Representations}

In this section, we start by letting $G$ be the general linear group of a complex vector space
$W$, and consider the canonical inclusion 
$G=\mathsf{GL}(W)\subset\mathsf{End}(W)\cong W\otimes W^{*}$.  We note here that $G$ is not a subvariety of 
$\mathsf{End}(W)$ although it is an algebraic group.  Also, in this section we again assume that all quivers are connected.

\subsection{General Additive Quiver Representations}

As additional consequences of the previous results, we will now relate
the moduli space of $G$-valued quiver representations to the usual
moduli spaces of quiver representations with a dimension vector which
is constant, with all entries equal to $d=\dim W$. To avoid confusion,
the latter will be sometimes called additive quiver representations.
For additive quiver representations we mainly follow the exposition
\cite{Rei2}.

\begin{defn}
Fix a quiver
$Q=(Q_{V},Q_{A})$. Recall that when $\mathbf{d}=(d_{v}|v\in Q_{V})$
is a constant dimension vector, so that $d_{v}=d$ for all $v\in Q_{V}$,
and $W$ is a complex vector space of dimension $d$, $\rep_{\mathbb{d}}(Q)\equiv \rep_{W}(Q):=\bigoplus_{Q_{A}}\mathsf{End}(W)$
is the space of quiver representations in $W$. We say that such a
(additive) quiver representation $x=(x_{a}|a\in Q_{A})\in \rep_{n}(Q)$
is \emph{invertible} if all its components $x_{a}$ are in $\mathsf{Aut}(W)=\mathsf{GL}(W)$.
\end{defn}

It is clear
that we have a canonical inclusion\begin{equation}
\varphi:\mathcal{F}_{G}(Q)\to \rep_{W}(Q)\label{eq:embedding}\end{equation}
obtained by mapping a $G=\mathsf{GL}(W)$-quiver representation to the corresponding
representation in $W$. The image of $\varphi$ is clearly the set
of invertible quiver representations in $W$, denoted $\rep_{W}^{inv}(Q)$,
and it is an open dense set in $\rep_{W}(Q)$. Note also that the dimension of the space of equidimensional 
additive quiver
representation (see \cite{Rei2}) coincides with the formula given
in Corollary \ref{cor:dimension} when $G=\GL(W)$.

Let $G_{W}(Q)=\mathsf{GL}(W)^{N_{V}}=\gauge_G(Q)$
and $\moduli_{W}^{inv}(Q):=\rep_{W}^{inv}(Q)\quot G_{W}(Q)$ denote the moduli space
of invertible quiver representations.

\begin{cor}\label{invertiblequiver}
Let $Q$ be
a quiver and $W$ be a vector space. The moduli space of invertible
quiver representations $\moduli_{W}^{inv}(Q)$ is isomorphic, as an affine
algebraic variety, to the character variety $\hom(F_{r},\mathsf{GL}(n,\mathbb{C}))\quot\mathsf{GL}(n,\mathbb{C})$,
where $r=b_{1}(Q)$.\end{cor}

\begin{proof}

This is just
a special case of Theorem \ref{charvarquiverequal} for $G=\mathsf{GL}(n,\C)$. 
\end{proof}

The natural
action (\ref{eq:action}) of $\gauge_G(Q)$ on $\mathcal{F}_{G}(Q)$
that we have been considering, corresponds naturally to the usual
action of $G_{W}(Q)$ on $\rep_{W}(Q)$. However, it is not generally
true that a closed orbit in $\mathcal{F}_{G}(Q)$ will be closed in
$\rep_{W}(Q)$ (under the embedding $\varphi$ above), as simple examples
show.  We now examine this phenomena as it is necessary to compare the algebro-geometric quotients \[
\mathcal{M}_{G}(Q)=\mathcal{F}_{G}(Q)\quot\gauge_G(Q)\quad\quad\mbox{and}\quad\quad \moduli_{W}(Q)=\rep_{W}(Q)\quot G_{W}(Q).\]

By a general theorem on representations of algebras \cite{A}, we know that a given
element $x\in \rep_{W}(Q)$ has a closed orbit if and only if it is semisimple,
that is, it is a direct sum of simple representations of $Q$ (those
without proper subrepresentations). Therefore the affine quotient
$\moduli_{W}(Q)=\rep_{W}(Q)\quot G_{W}(Q)$ parametrizes equivalence classes
of semisimple quiver representations in $W$. Denote the set of semisimple
representations by $\rep_{W}^{ss}$, and let $\fnc^{ss}_G=\fnc_G(Q)\cap\rep_W^{ss}.$

Let $G=\mathsf{GL}(W)$.  We can picture what we have been discussing, using a diagram that relates $\mathsf{GL}(W)$-quiver representations
and usual quiver representations (compare to the diagram in \cite{Rei2}):

\[
\begin{array}{ccc}
\mathcal{F}_{G}^{ss} & \to & \rep_{W}^{ss}\\
\downarrow &  & \downarrow\\
\mathcal{F}_{G} & \hookrightarrow & \rep_{W} \\
\downarrow &  & \downarrow \\
\mathcal{M}_{G} & \to & \moduli_{W} \end{array}\]

We will now study some of the maps involved in the above diagram.
\begin{defn}

We say that
a vertex $v\in Q_{V}$ is a \emph{sink }if all arrows containing $v$
point towards $v$, or equivalently, there are no arrows coming out
from $v$. Similarly, we say that $v\in Q_{V}$ is a \emph{source}
if all arrows containing $v$ point away from $v$. A source or a
sink will be called an \emph{end} of $Q$. An \emph{oriented cycle}
in a quiver is a path of arrows which beings and ends at the same
vertex, by always following the directions provided by the arrows. 
\end{defn}

One can characterize quivers with no ends as those having a path joining
any two points by always following the arrows.  In
other words, a quiver with no ends is ``oriented-path-connected''.
Also, note that quivers without oriented cycles always have ends.  On the other hand,
a quiver without any oriented cycles is called {\it acyclic}, and so a quiver with no ends
might otherwise be called {\it super-cyclic}.

A nice property of a representation of a super-cyclic quiver is that closed orbits are preserved under
the embedding $\varphi:\mathcal{F}_{G}(Q)\to R_{W}(Q)$.

\begin{thm}\label{thm:preserve-closed}
Let $Q$ be a super-cyclic quiver, $W$ be a complex vector space, and $G=\GL(W)$. Then, if
the $\gauge_G(Q)$-orbit of $f\in\mathcal{F}_{G}(Q)$ is closed then the image of $f$ under $\varphi:\mathcal{F}_{G}(Q)\to R_{W}(Q)$
has a closed $G_W$-orbit.
\end{thm}

\begin{proof}

To get a contradiction,
suppose that $x=\varphi(f)$ is not closed in $R_{W}(Q)$, but $f$
is closed in $\mathcal{F}_{G}(Q)$. Then, there exists a morphism
$\lambda:\mathbb{C}^{*}\to G_{W}$ such that $\lim_{t\to0}\lambda(t)\cdot x$
exists in $R_{W}(Q)\setminus R_{W}^{inv}(Q)$. This means that for
all arrows $a\in Q_{A}$, $\lim_{t\to0}\lambda_{h_{a}}(t)\ x(a)\ \lambda_{t_{a}}(t)^{-1}=M(a)$
exists in $\mathsf{End}(W)$, where $\lambda_{v}:\mathbb{C}^{*}\to G$ is the
component of $\lambda$ associated to $v\in Q_{V}$ (post-compose $\lambda$ with projection onto the $v^{\text{th}}$ component).

By composing
with the determinant map, for any vertex, we define the homomorphism
$\mu_{v}=\det\circ\lambda_{v}:\mathbb{C}^{*}\to\mathbb{C}^{*}$. Then,
there are weights $\alpha(v)\in\mathbb{Z}$ such that $\mu_{v}(t):=\det\lambda_{v}(t)=t^{\alpha(v)}$
for $t\in\mathbb{C}^{*}$.

The determinant
is a continuous morphism, so with $c=\det x(a)\neq0$, \begin{equation}
\det M(a)=\det x(a)\cdot\lim_{t\to0}\left(\mu_{h_{a}}(t)\cdot\mu_{t_{a}}(t)^{-1}\right)=
c\cdot\lim_{t\to0}t^{\alpha(h_{a})-\alpha(t_{a})}\label{eq:weight-limit}\end{equation}
Then, we obtain an inequality for weights $\alpha(h_{a})\geq\alpha(t_{a})$.
Now, since $Q$ has no ends, the arrow $a=a_{0}\in Q_{A}$ is in a
cycle formed by the sequence $(a=a_{0},a_{1}...,a_{n})$ such that
$h_{a}=t_{a_{1}},h_{a_{1}}=t_{a_{2}},\cdots,h_{a_{n}}=t_{a}$. So,
the repetition of the argument above for the other arrows in the same
cycle will imply that $\alpha(h_{a_{i}})\geq\alpha(t_{a_{i}})$ for
all $i=1,...,n$. So we conclude:\[
\alpha(h_{a})\geq\alpha(t_{a})=\alpha(h_{a_{n}})\geq\alpha(t_{a_{n}})=\alpha(t_{a_{n-1}})\geq\cdots\geq\alpha(t_{a_{1}})=\alpha(h_{a})\]
So the only solution is to have all weights equal in that cycle. Since
the quiver is connected, we conclude that $\alpha(v)$ is constant
for all $v\in Q_{V}$. 

On the other hand, we are assuming that $\lim_{t\to 0}\lambda(t)\cdot x\in R_{W}(Q)\setminus R_{W}^{inv}(Q)$.
This means that, for at least one arrow, say $b\in Q_{A}$, the corresponding
limit $\lim_{t\to0}\lambda_{h_{b}}(t)\ x(b)\ \lambda_{t_{b}}(t)^{-1}=M(b)$
belongs to $\mathsf{End}(W)\setminus \GL(W)$, so that $\det M(b)=0$. But this
forces $\alpha(h_{b})>\alpha(t_{b})$ in equation (\ref{eq:weight-limit})
which contradicts the equality of all the weights, asserted before,
and which was a consequence of our hypothesis.
\end{proof}

Recall our notation for character varieties:  $\X_r(\GL(n,\C)):=\hom(F_r,\GL(n,\C))\quot \GL(n,\C)$.

\begin{cor}[Density Theorem]
Let $Q$ be a super-cyclic quiver such that $r=b_1(Q)$, and suppose $W$ is a complex 
vector space of dimension $n$.
Then $\X_r(\GL(n,\C))\cong \moduli_{\GL(n,\C)}(Q)$ embeds into $\moduli_{W}(Q)$ as a dense subset.
\end{cor}

\begin{proof}
Apply the above theorem and the Equivalence Theorem to get an embedding.  The density is clear.
\end{proof}

The following proposition is a converse to this result.

\begin{prop}\label{pro:Not-closed}Let
$Q$ be a quiver and $v\in Q_{V}$ be a sink or a source. Let $a_{1},...,a_{k}$
denote the collection of all arrows connecting to $v$. Assume that
$x\in \rep_{W}(Q)$ is such that $x(a_{j})\neq 0$ for some $1\leq j\leq k$. Then, the orbit of
$x$ inside $\rep_{W}(Q)$ is not closed.
\end{prop}

\begin{proof}
Let $v$ be a sink. Then, by definition $h_{a_{j}}=v$, for all $j=1,...,k$.
The action of $g\in G_{W}$ on $x\in \rep_{W}(Q)$ takes the values\[
(g\cdot x)(a_{j})=g_{v}x(a_{j})g_{t_{a_{j}}}^{-1},\quad j=1,...,k.\]
If we take $g_{v}$ of the form $g_{v}=tI_{W}\in \mathsf{GL}(W)$, for some
$t\in\mathbb{C}^{*}$ and $I_{W}$ is the identity on $W$, then we
obtain\[
\lim_{t\to0}(g\cdot x)(a_{j})=0,\quad j=1,...,k.\]
 
Recall that the extended orbit of $x$ is the union of all orbits $G_{W}\cdot x'$ such that 
$\overline{G_{W}\cdot x}\cap\overline{G_{W}\cdot x'}\neq\emptyset$.

Observe that for any $\dim W \times \dim W$ matrix $x$, and $g,h\in \GL(W)$ if $gxh^{-1}=0_W$, then $x=g^{-1}0_Wh=0_W$.  The corollary of this 
observation is that any representation of a quiver with a non-zero marking on the $j^{\text{th}}$ arrow cannot be in the (non-extended) $G_W$-orbit of 
any representation having a marking of the zero matrix on the $j^{\text{th}}$ arrow.

Writing $x\in R_{W}(Q)$ as $(x(a_{1}),...,x(a_{k}),x(a_{k+1}),...,x(a_{N_{A}}))\in G^{N_{A}}$
we see that the additive representation $y=(0,...,0,x(a_{k+1}),...,x(a_{N_{A}}))\in R_{W}(Q)$
is not in the orbit of $x$ but it is in its closure. Note that $g_{v}=tI_{W}$
only acts on $x(a_{1}),...,x(a_{k})$ and not on other arrows. 

The same argument works for a source by considering the limit as $t\to\infty$.

Repeating this procedure for every sink and for every source, we end
up with an additive representation of the form $z=(0,...,0,x(a_{m}),...,x(a_{N_{A}}))$
where the only non-zero entries correspond to arrows which do not
connect to any sink or source. All these arrows belong to oriented
cycles. Because of this $G_{W}\cdot z$ is indeed a closed orbit
by Theorem \ref{thm:preserve-closed}.  We have showed that $z\in\overline{G_{W}\cdot x}$.
Since there is only one closed orbit in every extended orbit and $G_{W}\cdot x\neq G_{W}\cdot z$,
we conclude that the orbit $G_{W}\cdot x$ is not closed in $R_{W}(Q)$,
as wanted.

\end{proof}

In contrast to this situation, closed orbits are not preserved under $\varphi$
in quivers with ends, as we now show.

\begin{thm}
Let $Q$ be
a quiver and $G=\mathsf{GL}(W)$, and consider the inclusion $\varphi:\mathcal{F}_{G}(Q)\to \rep_{W}(Q)$
defined in \eqref{eq:embedding}. If $Q$ has at least one end, then
no orbit in $\varphi(\mathcal{F}_{G}(Q))=\rep_{W}^{inv}(Q)$ is closed.
In other words, for quivers with ends, every orbit of the form $\varphi(f)$
for $f\in\mathcal{F}_{G}(Q)$ is not closed.
\end{thm}

\begin{proof}
If $Q$ has one sink or source, $f\in\mathcal{F}_{G}(Q)$, and $x=\varphi(f)\in \rep_{W}$
then it is clear that $x(a)\not=0$, for all arrows, since $x(a)\in\mathsf{GL}(W)$.
The result then follows from Proposition \ref{pro:Not-closed}.
\end{proof}

\begin{cor}\label{densityproblem}
In a quiver
with ends, the semisimple additive representations do not form a dense
set.\end{cor}
\begin{proof}
Let $Q$ have
at least one end. Suppose that $\rep_{W}^{ss}$ was dense in $\rep_{W}$.
Since $\rep_{W}^{inv}\subset \rep_{W}$ is open and dense, the space $\rep_{W}^{ss}\cap \rep_{W}^{inv}$
would be non-empty. But we saw that any $\varphi(f)\in \rep_{W}^{inv}$
does not have a closed orbit. So it is not semisimple.
\end{proof}

\subsection{Unimodular Additive Quiver Representations}

There is another important situation when $G$-valued orbits will give rise to closed
orbits inside a variation of additive quiver representations.

\begin{defn}
Let $Q$ be
a quiver and $W$ a complex vector space. We say that $x\in \rep_{W}(Q)$
is a unimodular (or special) quiver representation if $x(a)\in \mathsf{SL}(W)=\{X\in \mathsf{GL}(W):\det X=1\}$
for every arrow $a\in Q_{A}$. Note that this restriction is well
defined since the determinant is independent of the choice of basis
for $W$. Unimodular quiver representations of $Q$ will be denoted by
$\rep_{W}^{1}(Q)$.\end{defn}

We now consider
$G$-valued quiver representations with $G=\mathsf{SL}(W)$. Naturally, we
have the same inclusion\[
\varphi:\mathcal{F}_{G}(Q)\to \rep_{W}(Q)\]
still denoted by the same letter. On both spaces of representations
we now have the action of the group $S_{W}=\mathsf{SL}(W)^{N_{V}}$, acting
in the same way. The image under $\varphi$ will be exactly 
$\rep_{W}^{1}(Q)\subset \rep_{W}^{inv}(Q)\subset \rep_{W}(Q)$.

\begin{prop}

Let $G=\mathsf{SL}(W)$
and $Q$ be any quiver. Then $f\in\mathcal{F}_{G}$ has a closed
orbit if and only if $x=\varphi(f)$ has a closed orbit in $\rep_{W}(Q)$.\end{prop}
\begin{proof}
If $f$ has
a closed orbit $O_{f}$ then, since $\varphi$ is an embedding of
$\mathcal{F}_{G}(Q)$ as a subvariety, the orbit $O_{x}$ of $x=\varphi(f)$
is closed in $\rep_{W}(Q)$. Conversely, if $x=\varphi(f)$ has a closed
orbit in $\rep_{W}$, then the orbit is also closed in the closed subset
$\rep_{W}^{1}$, so $f$ has a closed orbit in $\mathcal{F}_{\mathsf{SL}(W)}$.\end{proof}

In particular, this shows that, for $G = \SL(W)$, we have
$$\moduli_G := \fnc_G \quot S_W \cong \rep^1_W \quot S_W = \moduli_W^1.$$

\begin{thm}\label{unimodularsubvariety}
Let $Q$ be
a super-cyclic quiver. Then, the moduli space $\moduli_{W}^{1}:=\rep_{W}^{1}\quot S_{W}$
is naturally a subvariety of $\moduli_{W}=\rep_{W}\quot G_{W}$.\end{thm}

\begin{proof}
Let $x,x'\in \rep_{W}^{1}$
be two unimodular quiver representations. Then, we claim that $x$ and
$x'$ are in the same $S_{W}$ orbit if and only if they are in the
same $G_{W}$ orbit. Clearly, the $S_{W}$-orbit sits inside the
$G_{W}$-orbit. Now suppose $(x_{1},x_{2},...,x_{n})\in \mathsf{SL}(W)^{n}$
are the values of $x\in \rep_{W}^{1}$ on the arrows of an oriented cycle
of $Q$ of length $n$, and that $(x'_{1},x'_{2},...,x'_{n})$ are the corresponding
values of $x'\in \rep_{W}^{1}$. Assuming $x$ and $x'$ are in the same
$G_{W}$-orbit there are matrices $g_{1},g_{2},...\in \mathsf{GL}(W)$ such
that\[
(x'_{1},x'_{2},...,x'_{n})=(g_{1}x_{1}g_{2}^{-1},g_{2}x_{2}g_{3}^{-1},...,g_{n-1}x_{n}g_{n}^{-1}),\]
Now, since $x_{j},x_{j}'$ have determinant one, we conclude that
$\det g_{1}=\det g_{2}=...=\det g_{n}$. Also, since $Q$ has no ends,
any arrow is in one of these oriented cycles. For an oriented cycle $C$ let $A(C)$ be the collection of arrows in $C$.  
Since the quiver is connected and has has no ends, there is a collection
of cycles $C_1,...,C_k$ so that $\cup_{i=1}^k A(C_i)=Q_A$ and $C_i\cap C_{i+1}\not=\emptyset$.  Thus, we can just rescale
all the terms of $g\in G_{W}$ by the same scalar in order to obtain
an element $g'\in S_{W}$ such that $g'\cdot x=x'$ as wanted.

This shows
that we have a well defined inclusion of orbit spaces, \[
\moduli_{W}^{1}=\rep_{W}^{1}\quot S_{W}\cong\rep_{W}^{1}\quot G_{W}\subset \rep_{W}\quot G_{W}=\moduli_W,\]
as wanted.\end{proof}

\begin{cor}\label{quivercharvar}
Let $Q$ be a super-cyclic quiver, $d=\dim W$ and $b_{1}(Q)=r$. Then, the $\SL(d,\mathbb{C})$-character
variety of a free group of rank $r$ is naturally a subvariety of
the usual moduli of quiver representations with fixed dimension vector
$d$.\end{cor}

\begin{proof}
Let $\X_r(G):=\hom(F_r,G)\quot G$ be the $G$-character variety of the free group of rank $r$, denoted $F_r$, with $G=\mathsf{SL}(W)=\SL(d,\C)$.  Then 
Theorem \ref{charvarquiverequal} and the preceding proposition imply that we have a commutative diagram,\[
\begin{array}{cccccc}
&\mathcal{F}_{G} & \cong & \rep_{W}^{1} & \subset & \rep_{W}\\
&\downarrow &  & \downarrow &  & \downarrow\\
\X_r(G)\cong& \mathcal{M}_{G} & \cong & \moduli_{W}^{1} & 
\subset & \moduli_{W}.\end{array}\]
\end{proof}

\begin{thm}[Embedding Theorem]\label{embedcharvarintovectquiv}
Let $G$ be a complex reductive Lie group, and $\Gamma$ be a finitely presented group with $r$ generators.  Then there exists a number 
$d\geq 0$ so that $\X_\Gamma(G)$ embeds as a subvariety in $\moduli_{W}(Q)$ for any quiver $Q$ with no ends, $b_{1}(Q)=r$, and $W$ a complex 
vector space of dimension $d$.
\end{thm}

\begin{proof}
Since $G$ is algebraic it admits a faithful linear representation and hence admits a faithful unimodular representation 
into $\SL(W)$.  Since we can always find an epimorphism $F_r \to \Gamma$ for some free group $F_r$, we conclude 
$\X_\Gamma(G)\subset \X_r(G)$.

The result follows then from the previous corollary:
$$\X_\Gamma(G)\subset \X_r(G)\subset \X_r(\SL(W))\cong  \mathcal{M}_{\SL(W)}(Q)  \cong  \moduli_{W}^{1}(Q) 
\subset  \moduli_{W}(Q).$$
\end{proof}

\subsection{Relationship with Toric Geometry}\label{toric}

In this appendix we relate some of our discussion to toric geometry.  We use 
\cite{CLS} as a reference.

\subsubsection{$G$-representations of quivers for Abelian $G$}

Let $Q=(Q_{V},Q_{A})$ be a quiver and $G$ a Lie group. We have been
studying the action of the gauge group $\gauge_G(Q):=\fun(Q_{V},G)\cong
G^{\#Q_{V}}$
on the group of $G$-valued representations of $Q$,
$\mathcal{F}_{G}(Q)=\fun(Q_{A},G)\cong G^{\#Q_{A}}$.
In order to try to generalize this type of action, suppose now we
are given maps $\mu,\nu:Q_{A}\to\mathbb{N}_0$ (weight maps) which associate
integer weights $\mu_{a},\nu_{a}\in\mathbb{N}_0$ for each arrow $a\in Q_{A}$
and consider the map\begin{eqnarray*}
\psi:\gauge_G(Q)\times\mathcal{F}_{G}(Q) & \to & \mathcal{F}_{G}(Q)\\
(g,f) & \mapsto &
\psi(g,f)(a):=g(h_{a})^{\mu_{a}}f(a)g(t_{a})^{-\nu_{a}}.\end{eqnarray*}
The following is straightforward.
\begin{lem}
Suppose at least one of the integers $\mu_{a},\nu_{a}$ is not $0$ or
$1$. Then, the map $\psi$ defines a (left) action of $\gauge_G(Q)$
on $\mathcal{F}_{G}(Q)$ if and only if $G$ is Abelian.\end{lem}
\begin{proof}
This follows from simple computations. By definition, for each arrow
$a\in Q_{A}$,\[
\psi(g,\psi(\tilde{g},f))(a)=g(h_{a})^{\mu_{a}}\tilde{g}(h_{a})^{\mu_{a}}f(a)\tilde{g}(t_{a})^{-\nu_{a}}g(t_{a})^{-\nu_{a}},\]
but on the other hand:\[
\psi(g\tilde{g},f)(a)=\left(g(h_{a})\tilde{g}(h_{a})\right)^{\mu_{a}}f(a)\left(g(t_{a})\tilde{g}(t_{a})\right)^{-\nu_{a}}.\]
So, the (left) action property requires that
$g(h_{a})\tilde{g}(h_{a})=\tilde{g}(h_{a})g(h_{a})$
for any arrow $a$ with $\mu_{a}\notin\{0,1\}$.
\end{proof}
According to the lemma, to work with this more general map, we have
to consider Abelian groups $G$. In the context of affine reductive
groups, the most general such group $G$ of dimension $n$ is the product of an
algebraic torus $(\mathbb{C}^{*})^{n}$ with a finite Abelian
group.

For fixed choices of quiver $Q$, weight maps $\mu$, $\nu$, and Abelian reductive
group $G$, we will refer to the corresponding orbit space
$\mathcal{F}_{G}(Q)\quot_{\mu,\nu}\gauge_G(Q)$
as the moduli space of weighted $G$-quiver representations.

\subsubsection{The case when $G=\mathbb{C}^{*}$}

To address the relationship with toric geometry, we suppose from now
on that $G=\mathbb{C}^{*}$.

Every toric variety $X$ can be constructed
as an affine GIT quotient of the form\begin{equation}
X=(\mathbb{C}^{n}-Z)\quot H\label{eq:toric-quotient}\end{equation}
where $Z$ is the so-called exceptional set, and $H$ is an Abelian
reductive group (see \cite{CLS}, page 210, Theorem 5.1.11).

By the above cited theorem, we have that
$H\cong\left(\mathbb{C}^{*}\right)^{k}$ for some $k\leq n$ in the
realization (\ref{eq:toric-quotient}).

Let $Y\subset\mathbb{C}^{n}$ be the union of all coordinate hyperplanes.
From the very definition of the exceptional set, it is clear that
$Z\subset Y\subset\mathbb{C}^{n}$. Then, we can let
$X^{\circ}:=(\mathbb{C}^{n}-Y)\quot H\subset X$
and call this the {\it big open cell} of the toric variety $X$.

\begin{prop}
For any super-cyclic quiver $Q$, and any weight
maps $\mu,\nu:Q_{A}\to\mathbb{N}_0$ there exists an affine toric variety $X$ such that
$\mathcal{F}_{\mathbb{C}^{*}}(Q)\quot\mathcal{G}_{\mathbb{C}^{*}}(Q)$
is naturally isomorphic with the big open cell of $X$.
\end{prop}

\begin{proof}
Apply the Density Theorem with $\GL(1,\C)=\C^*$ to the context of weighted quiver actions.
\end{proof}

\appendix

\section{Kempf-Ness Set and Strong Deformation Retraction}

The following argument is virtually identical to the argument we gave in \cite{FlLa}. The ideas directly generalize to the context of $G$-valued quiver representations. 

Let $\V_Q=\fun(Q_A,\mathfrak{gl}(n,\C))=\mathfrak{gl}(n,\C)^{N_A}=\C^{N_A n^2}$.  Then 
$\mathcal{F}_G(Q)=G^{N_A}\subset \V_Q$ and the action of both $\mathcal{G}_K\subset \mathcal{G}_G$ naturally extends to $\V_Q$.

For any $x,y\in \mathfrak{gl}(n,\C)$ define $\langle x,y\rangle=\tr(xy^*)$ where $y^*$ is the conjugate transpose of $y$.  We will denote transpose
by a dagger and complex conjugation by a bar, so $y^*=\overline{y}^\dagger$.  This form is thus Hermitian.  
The form is also $K$-conjugate invariant since $K$ is closed under inversion and this is equivalent to taking transposes in an orthogonal 
representation, and complex conjugation is trivial in a real representation.  We also note that conjugate transpose is defined on $G$ since $G$ is the 
complex zeros of real polynomials defining $K$ which implies it is conjugate invariant; since it embeds in a complex orthogonal group it is transpose
invariant.

Define on $\V_Q$ the Hermitian inner product\[
\left\langle f_1,f_2\right\rangle =\sum_{a\in Q_A}\left\langle f_1(a),f_2(a)\right\rangle .\]

We now show that this form is $\mathcal{G}_K$ invariant:  
\begin{eqnarray*}
\langle k\cdot f_1,k\cdot f_2\rangle &=&\sum_{a\in Q_A}\tr(k(h_a)f_1(a)k(t_a)^{-1}(k(h_a)f_2(a)k(t_a)^{-1})^*)\\
                                                &=&\sum_{a\in Q_A}\tr(k(h_a)f_1(a)k(t_a)^{-1}(k(t_a)^\dagger)^\dagger f_2(a)^*k(h_a)^\dagger)\\
                                                &=&\sum_{a\in Q_A}\tr(k(h_a)f_1(a)k(t_a)^{-1}k(t_a) f_2(a)^*k(h_a)^{-1})\\
                                                &=&\sum_{a\in Q_A}\tr(k(h_a)f_1(a)f_2(a)^*k(h_a)^{-1})\\
                      &=&\sum_{a\in Q_A}\tr(f_1(a)f_2(a)^*)=\sum_{a\in Q_A}\left\langle f_1(a),f_2(a)\right\rangle=\left\langle f_1,f_2\right\rangle.
\end{eqnarray*}

Note that $\mathrm{Lie}(\mathcal{G}_K)=\mathfrak{k}^{N_V}$, and recall that the Kempf-Ness set is:
$$\mathcal{KN}_Q=\{f\in \V_Q\ |\ \langle u\cdot f,f\rangle=0\ \mathrm{for}\ \mathrm{all}\ u\in\mathfrak{k}^{N_V}
\}\cap \fnc_G(Q).$$

To make sense of this we need to determine how $\mathrm{Lie}(\mathcal{G}_K)\subset \mathrm{End}(\V_Q)$ acts.  Since the action is differential,
it suffices to consider the action in one $f(a)$ for $f\in\mathcal{F}_G$.  Fix $u\in\mathrm{Lie}(\mathcal{G}_K)$. Let $k_t(v)=e^{-tu(v)}$  
be a path starting at the identity in the direction $-u(v)\in \mathfrak{k}$.  Then the infinitesimal quiver action on the component corresponding 
to $a\in Q_A$ is given by: $$\frac{d}{dt}\bigg|_{t=0}(k_t(h_a)f(a)k_t(t_a)^{-1})=f(a)u(t_a)-u(h_a)f(a).$$  

Thus, the Kempf-Ness set consists of the quiver representations that satisfy:
$$0=\langle u\cdot f,f\rangle=\sum_{a\in A}\tr(f(a)u(t_a)f(a)^*-u(h_a)f(a)f(a)^*),$$ for all $u\in \mathrm{Lie}(\mathcal{G}_K)$.

Our next task is to show that whenever $f\in \fnc_K(Q)$, that this equation is in fact satisfied.  Precisely, 

\begin{prop}\label{kempfnesssection}
$\fnc_K(Q)\subset \mathcal{KN}_Q$.
\end{prop}

\begin{proof} 
Indeed, when $f(a)\in K$ we have $f(a)^*=f(a)^\dagger=f(a)^{-1}$ since $K\subset \mathrm{O}(n,\R)$ and likewise $u(v)^\dagger =-u(v)$ since 
$\mathfrak{k}\subset \mathfrak{o}(n,\R)$.  Note that this implies that $\tr(u(v))=0$ for all vertices $v\in Q_V$.

Thus 
\begin{eqnarray*}
\langle u\cdot f,f\rangle&=&\sum_{a\in Q_A}\tr(f(a)u(t_a)f(a)^*-u(h_a)f(a)f(a)^*)\\
 &=&\sum_{a\in Q_A}\tr(f(a)u(t_a)f(a)^{-1})-\tr(u(h_a)f(a)f(a)^{-1})\\
&=&\sum_{a\in Q_A}\tr(u(t_a))-\tr(u(h_a))\\
&=&0,
\end{eqnarray*}
since $u$ is a tuple of traceless matrices.
\end{proof}

The inclusions $\fnc_K(Q)\subset \mathcal{KN}_Q\subset \fnc_G(Q)$ from Proposition \ref{kempfnesssection} induce a continuous injection 
$\sigma:\fnc_G(Q)\quot\gauge_G \cong \mathcal{KN}_Q/\gauge_K\to \fnc_G(Q)/\gauge_K$.

\begin{proof}[Alternative Proof of Theorem \ref{defretractthm}]
It is elementary to establish that the following diagram is
commutative: \[
\xymatrix{\fnc_G(Q)/\gauge_K\ar[dd]^{\pi_{G/K}}\ar[rr]^{\Phi_{t}} &  & \fnc_G(Q)/\gauge_K\ar[dd]^{\pi_{G/K}}\\
& \moduli_K(Q)\ar@{^{(}->}[ul]_{\mathrm{id}}\ar@{^{(}->}[ur]^{\mathrm{id}}\ar@{^{(}->}[dl]^{\iota}\ar@{^{(}->}[dr]_{\iota}\\
\moduli_G(Q)\ar@{-->}@/^3pc/[uu]^{\sigma} &  & \moduli_G(Q)}
\]

Define $\Phi_{t}^{\sigma}=\pi_{G/K}\circ\Phi_{t}\circ\sigma$. Then
since all composite maps are continuous, so is $\Phi_{t}^{\sigma}$.
We now verify the other properties of a strong deformation retraction.
Firstly, $\Phi_{0}^{\sigma}$ is the identity since $\Phi_{0}=\mathrm{id}$
and $\pi_{G/K}\circ\sigma=\mathrm{id}$.

Next, we show $\Phi_{1}^{\sigma}$ is into $\iota(\moduli_K(Q))$.
Since $\Phi_{1}(\fnc_G(Q)/\gauge_K)\subset \moduli_K(Q)$, it follows
$\mathcal{I}:=\Phi_{1}\left(\sigma(\moduli_G(Q))\right)\subset \moduli_K(Q)$.
Moreover, $\pi_{G/K}=\iota$ on $\moduli_K(Q)$, so $\Phi_1^{\sigma}(\moduli_G(Q))=\pi_{G/K}(\mathcal{I})=\iota(\mathcal{I})\subset\iota(\moduli_K(Q))$.

Lastly, we show that for all $t$, $\Phi_{t}^{\sigma}$ is the identity
on $\iota(\moduli_K(Q))$. Indeed, commutativity of the above diagram implies that $(\sigma\circ\iota)=\mathrm{id}$
on $\moduli_K(Q)\subset \fnc_G(Q)/\gauge_K$. Also, for all $t$, $\Phi_{t}$
is the identity on $\moduli_K(Q)$. Lastly, using the fact that $\pi_{G/K}=\iota$
on $\moduli_K(Q)$, we have for any point $[f]\in \moduli_K(Q)$,
\[
\iota([f])\mapsto\sigma(\iota([f]))=[f]\mapsto\Phi_{t}([f])=[f]\mapsto\pi_{G/K}([f])=\iota([f]),\]
 as was to be shown. 
\end{proof}

\bibliographystyle{amsalpha}

\def\cdprime{$''$} \def\cdprime{$''$} \def\cprime{$'$} \def\cprime{$'$}
  \def\cprime{$'$} \def\cprime{$'$}
\providecommand{\bysame}{\leavevmode\hbox to3em{\hrulefill}\thinspace}
\providecommand{\MR}{\relax\ifhmode\unskip\space\fi MR }
\providecommand{\MRhref}[2]{%
  \href{http://www.ams.org/mathscinet-getitem?mr=#1}{#2}
}
\providecommand{\href}[2]{#2}

\end{document}